\documentclass[12pt,a4paper,twoside]{article}
\usepackage{amsmath}
\usepackage{amsthm}
\usepackage{amssymb}
\usepackage{verbatim}
\usepackage{booktabs}
\usepackage{caption}
\usepackage{subfig}
\usepackage{tabulary}
\usepackage{mathrsfs}
\usepackage{thmtools}
\usepackage{enumerate}

\usepackage[pdftex]{graphicx}
\usepackage{float}
\floatstyle{boxed}
\restylefloat{figure}

\usepackage{soul}
\usepackage{color}
\usepackage[usenames,dvipsnames]{xcolor}
\usepackage{tikz}
\usetikzlibrary{shapes,snakes}
\usepackage{multicol}
\usepackage{tkz-tab, caption, latexsym}
\usepackage{pgfplots}

\definecolor{light-gray}{gray}{0.925}

\newenvironment{customlegend}[1][]{%
    \begingroup
    \csname pgfplots@init@cleared@structures\endcsname
    \pgfplotsset{#1}%
}{%
    \csname pgfplots@createlegend\endcsname
    \endgroup
}%

\def\addlegendimage{\csname pgfplots@addlegendimage\endcsname}

\DeclareMathOperator*{\argmax}{argmax}

\newcommand{\R}{{\mathbb{R}}}

\title{Characterizing and Finding \\
the Pareto Optimal Equitable Allocation
\\
of Homogeneous Divisible Goods
\\
Among Three Players}
\author{Marco Dall'Aglio, Camilla Di Luca, Lucia Milone \\ LUISS University \\Rome, Italy \\ \texttt{mdallaglio@luiss.it} , \texttt{cdiluca@luiss.it}, \texttt{lmilone@luiss.it} }

\date{May 30, 2016}

\theoremstyle{plain}
\newtheorem{theo}{Theorem}
\newtheorem{defin}{Definition}
\newtheorem{prop}{Proposition}

\newtheorem{lem}{Lemma}
\theoremstyle{remark}

\begin{document}

\tikzstyle{mybox}=[draw=black, very thick, rectangle, inner ysep=5pt, inner xsep=5pt]
\tikzstyle{fancytitle} =[fill=lightgray, text=black]

\renewcommand\thmcontinues[1]{Continued}

\maketitle

\begin{abstract}
We consider the division of a finite number of homogeneous
divisible items among three players. Under the assumption that each player assigns a positive value to every item, we characterize the optimal allocations and we develop two
exact algorithms for its search. Both the characterization and the algorithm are based on the tight relationship two geometric objects of fair division: the Individual Pieces Set (IPS) and the Radon-Nykodim Set (RNS).
\end{abstract}


\section{The Problem}

This paper investigates the optimal allocation problem for a finite number $m$ of divisible and homogeneous objects, with $M=\{1,2,\ldots,m\}$, $m \in \mathbb{N}$, diputed among three players. Players will be usually denoted as $N= \{1,2,3\}$, but roman numbers $I$, $II$ and $III$ will be employed in the pictures. 

We write the matrix of evaluation as $(a_{ij})_{i \in N; j \in M}$, where each entry $a_{ij}$ tells us the value that player $i \in N$ assigns to item $j \in M$. We assume that utilities are
\begin{description}
\item[normalized] $\sum_{j \in M} a_{ij}=1$, $\forall i \in N$\\
i.e., utilities attached to the $q$ goods sum up to 1 for each player; and
\item[linear] if player $i$ gets share $t_j \in [0,1]$ of item $j$ and share $t_k \in [0,1]$ of item $k$, she gets a total utility of $t_j a_{ij} + t_k a_{ik}$.
\end{description}

Let $\mathbf{X}=\{x_{ij}\}_{i \in N; j \in M}$, $x_{ij}\geq 0$, $\forall i \in N, j \in M$ be an allocation matrix, with $\sum_{i\in N} x_{ij} = 1$, $\forall j \in M$ and $\mathbf{X} \in \mathcal{X}$, where $\mathcal{X}$ is the set of all possible allocations matrices. 
Let us label with $\hat{\mathbf{X}}$ any integer allocations where $x_{ij} \in \{0,1\}$, $\forall i \in N, j \in M$, and with $\hat{\mathcal{X}}$ the set of such allocations. Any integer allocation can be equivalently described by a vector $\mathbf{x}= (x_j)_{j \in M}$ such that each component $x_j=I,II$ or $III$ depending on whether $x_{1j}=1$, $x_{2j}=1$ or $x_{3j}=1$. Define now, for any $\mathbf{X} \in \mathcal{X}$, 
\[
a(\mathbf{X})=\left( \sum_{j \in M} a_{ij} x_{ij} \right)_{i \in N}.
\]
It is a vector in which each entry tells us, for any player, the total value that she derives from the given allocation $\mathbf{X}$.

We are going to search for an allocation $\mathbf{X}^*$ which simultaneously satisfies
\begin{description}
\item[(Strong) Pareto Optimality (PO)] There is no other allocation $\mathbf{X}' \in \mathcal{X}$
such that [
\[
a_i(\mathbf{X}') \geq a_i(\mathbf{X}^*) \qquad i \in N
\]
with strict inequality for at least one player.
\item[Equitability (EQ)] $a_1(\mathbf{X}^*) = a_2(\mathbf{X}^*) = a_3(\mathbf{X}^*)$
\end{description}
The proposed allocation coincides with the Kalai-Smorodinsky solution (see  \cite{ks75} and \cite{k77}) for bargaining problems. 

A well-known procedure for two players is the Adjusted Winner (AW) (see \cite{bt96} and \cite{bt99}). This procedure returns allocations that are not only PO-EQ but also envy-free. 

Throughout the rest of the work we are going to consider the following simplifying assumption:

{\bf Mutual absolute continuity (MAC)}. Each player assigns a positive value to any item
\[
a_{ij} > 0 \qquad \mbox{for any } i \in N \mbox{ and } j \in M
\]
When MAC holds a PO-EQ allocation always exists, and it coincides with the maxmin allocation defined by
\[
\mathbf{X}^* \in \argmax_{\mathbf{X} \in \mathcal{X}} \left\{ \min_{i \in N} a_i(\mathbf{X}) \right\}
\]

\section{Geometrical Framework}
We are now going to review two geometric structures that are useful for the analysis of PO and PO-EQ allocations. 
First of all we characterize the  PO allocations.

\begin{theo} (Theorem 1, \cite{bz97}, Proposition 4.3 \cite{d01})
Under MAC, an allocation is PO iff, for some $\gamma = (\gamma_1,\gamma_2,\gamma_3) \in \stackrel{\circ}{\Delta}_2$ the following holds: 
\begin{equation}
\label{po_gamma}
x_{ik} > 0 \qquad \mbox{if} \quad \gamma_i a_{ik} \geq \gamma_j a_{jk} \mbox{ for any }i,j \in N, \quad k\in M.
\end{equation}
We denote with $\mathbf{X}^{\gamma}=\{x_{ik}^{\gamma}\}$ any allocation satisfying \eqref{po_gamma}
\end{theo}

\subsection{The Partition Range}
We consider the Individual Pieces Set $IPS \subset \R^3$ (see \cite{b05}), also known as Partition Range, defined as follows
\[
IPS = \left\{  a(\mathbf{X}): \mathbf{X} \in \mathcal{X}  \right\} .
\]
\begin{prop}
\label{chullP}
$IPS = \mathrm{conv} \left( a(\hat{\mathbf{X}}) : \hat{\mathbf{X}} \in \hat{\mathcal{X}} \right)$.
\end{prop}
\begin{proof} First of all we show that all the extreme points of $IPS$, i.e.\ all points in $IPS$ which are not interior points of any segment lying in $IPS$, correspond to integer allocations of the goods.

Argue by contradiction, and suppose that an extreme point of $IPS$ corresponds only to noninteger allocations of goods. Let $\mathbf{X}_n = \{x_{ij}^n\}$ be any such allocation. Since $\mathbf{X}_n$ is non-integer, there must exist a good $j_0 \in M$ and two players $i_1,i_2 \in N$ such that, for some $\delta > 0$
\[
\delta < x_{i_1,j_0},x_{i_2,j_0} < 1 - \delta
\]
The following must also hold:
\begin{equation}
\label{aij_notzero}
a_{i,j_0} \neq 0 \mbox{ for at least an } i \in \{i_1,i_2\}
\end{equation}
In fact, assuming $a_{i_1,j_0} = a_{i_2,j_0} = 0$ we can replace $\mathbf{X}_n$ with another allocation $\mathbf{X}_t$ which is integer in the good $j_0$ and such that $a(\mathbf{X}_n)=a(\mathbf{X}_t)$. The argument can be replicated to other goods to conclude that \eqref{aij_notzero} holds. Without loss of generality, we assume $a_{i_1,j_0} \neq 0$ an consider two other allocations.
\begin{gather*}
\mathbf{X_+} = \begin{cases}
x_{ij}^n & j \neq j_0 \mbox{ or }( j=j_0 \mbox{ and } i\neq i_1,i_2)
\\
x_{ij}^n + \delta & j=j_0 \mbox{ and } i=i_1
\\
x_{ij}^n - \delta & j=j_0 \mbox{ and } i=i_2
\end{cases}
\\
\mathbf{X_-} = \begin{cases}
x_{ij}^n & j \neq j_0 \mbox{ or }( j=j_0 \mbox{ and } i\neq i_1,i_2)
\\
x_{ij}^n - \delta & j=j_0 \mbox{ and } i=i_1
\\
x_{ij}^n + \delta & j=j_0 \mbox{ and } i=i_2
\end{cases}
\end{gather*}
Now $a_{i_1}(\mathbf{X}_+) - a_{i_1}(\mathbf{X}_-) = 2 \delta a_{i_1,\j_0} \neq 0$. Therefore $a(\mathbf{X}_+) \neq a(\mathbf{X}_-)$ and $a(\mathbf{X}_n)$ is the midpoint of the segment $\left[  a(\mathbf{X}_+) , a(\mathbf{X}_-)\right]$, yielding a contradiction. By Carath\'eodory's Theorem (see for instance \cite{dgk63}), and the fact that $IPS \subset \mathbb{R}^3$, every point of $IPS$ is the convex combination of at most 4 extreme points of $IPS$
\end{proof}

The value of a PO-EQ allocation is the common coordinate of the intersection between the egalitarian ray, i.e.\ the orthant of the first quadrant in $\R^3$, and the upper surface of $IPS$, denoted as the Pareto Boundary, thereon PB. Proposition \ref{chullP} shows that PB is composed of faces, denoted Pareto faces (PF). MAC implies that no Pareto face is parallel to any of the coordinate axes.

If we consider the partition range from above, Finding the PO-EQ allocation amounts to finding the face of the PB that contains the egalitarian ray (actually more than one face may be involved if the egalitarian ray "hits" an edge, or coincides with an integer allocation), and then find the allocation of the Pareto face which yields the optimal value (Figure \ref{fig1}(a)). 

Consider, for any $x=(x_1,x_2,x_3) \in \R^3_+$, the normalizing operator
\[
N(x)= \left( \frac{x_1}{s(x)},\frac{x_{2}}{s(x)},\frac{x_{3}}{s(x)} \right) \qquad \mbox{with } s(x) = x_{1}+x_{2}+x_{3}
\]
and define the Normalized Pareto Boundary, thereon NPB, as a $\Delta_2$ simplex such that
\[
NPB=\left\{ N(x): x \in PB \right\}
\]
Then, the Pareto Faces partition the set NPB, and finding the PO-EQ allocation amounts to finding the allocation corresponding to the center $(1/3,1/3,1/3)$ on NPB (Figure \ref{fig1}(b)).
\begin{figure}[htp]
\centering
\includegraphics[scale=0.50]{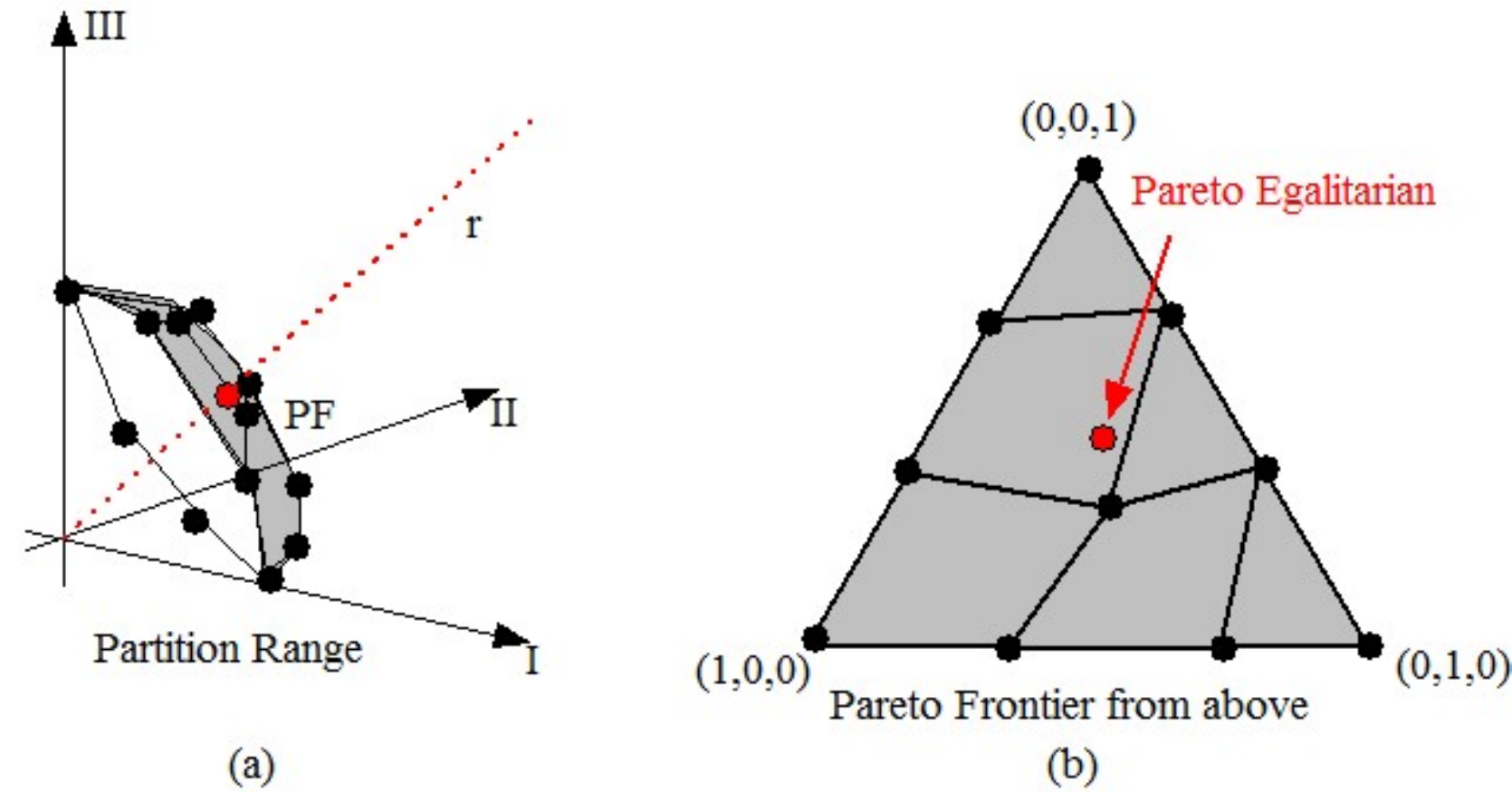}
\caption{(a) The red dot indicates the value of the PO-EF allocation. (b) The same picture from above }\label{fig1}
\end{figure}

To find the PO-EQ allocations we will employ  the following result valid for general fair division problems (any number of players, completely divisible and non-homogeneous goods).

\begin{theo} 
\label{th_dallaglio} (Proposition 6.1 in \cite{d01})
Consider the following function $g:\Delta_2 \to [0,1]$
\[
g(\gamma) = \sum_{i \in N} \gamma_i a_i(\mathbf{X}^{\gamma}) \qquad \gamma = (\gamma_1,\gamma_2,\gamma_3) \in \Delta_2
\]
with $\mathbf{X}^{\gamma}$ PO allocation associated to $\gamma$ by \eqref{po_gamma}.
Then
\begin{enumerate}[(i)]
\item The hyperplane
\[
\mathcal{H}(\gamma) = \left\{ (x_1,x_2,x_3) \in \mathbb{R}^3: \sum_{i \in N} \gamma_i x_i = g(\gamma)  \right\}
\]
supports $IPS$ at the point $(a_1(\mathbf{X}^{\gamma}),a_2(\mathbf{X}^{\gamma}),a_3(\mathbf{X}^{\gamma}))$, i.e.
\begin{gather*}
(a_1(\mathbf{X}^{\gamma}),a_2(\mathbf{X}^{\gamma}),a_3(\mathbf{X}^{\gamma})) \in \mathcal{H}(\gamma) 
\\
\mbox{ and }\quad  \sum_{i \in N} \gamma_i y_i \leq g(\gamma) \qquad \forall (y_1,y_2,y_3) \in \mathcal{P}
\end{gather*}
\item The hyperplane $\mathcal{H}$ intersects the egalitarian ray at the point $g(\gamma)(1,1,1)$
\item The function $g(\cdot)$ is convex, and for any of its minimizing points $\gamma^*$ the hyperplane $\mathcal{H}(\gamma^*)$ supports $IPS$ at a set of points containing the PO-EQ allocation.
\end{enumerate}
\end{theo}
In \cite{d01} an algorithm that returns the leximin allocation is described. This can be adapted to return the PO-EQ allocation in the present situation:
\begin{enumerate}[1.]
\item Find $\gamma^*$, an absolute minimum for $g$
\item Find the Pareto face corresponding to $\mathcal{H}(\gamma^*)$
\item Find the equitable allocation within the Pareto face
\end{enumerate}
In order to fully adapt the algorithm to the present situation, we need to better characterize the Pareto faces.

\subsection{The Radon-Nykodim set}
Figure \ref{fig1}(b) shows that the PB can be represented as a 2-dimensional simplex. We will now consider another 2-dimensional simplex, due to Weller \cite{w85} and extensively investigated by Barbanel \cite{b05}, that enables us to represent the items, the efficient partitions and the faces of PB  into a single geometric figure. Following \cite{b05}, we refer to the Radon-Nikodyn set, thereon RNS, to define this new simplex.

Each vertex of RNS represents a player. We next plot the single items into RNS by considering the normalized vectors of evaluations of the single items
\[
a^n_j=N(a_{\cdot j}) \qquad j \in M
\]
The normalized coordinates of all objects are plotted on a 2-dimensional simplex where each vertex represents a player. Under MAC, $a^n_j \in \stackrel{\circ}{\Delta}_2$ for each $j \in M$.
\begin{defin}
For each point $\beta=(\beta_1,\beta_2,\beta_3) \in RNS$ in the simplex, consider the lines joining $\beta$ with each vertex; we denote as \textbf{disputing segments} the half open segments on those lines from $\beta$ to the opposite side of each vertex, with $\beta$ excluded.
\end{defin}
\begin{defin}
For each $\beta \in \stackrel{\circ}{RNS}$, we derive the following {\bf Pareto Allocation Rule} after $\beta$, thereon $PAR(\beta)$, which delivers one or more PO allocations under MAC (see theorem 10.9 in \cite{b05}). The disputing segments of $\beta$ divide the simplex in three parts, each a neighborhood of a vertex. The objects in each neighborhood are assigned to the player associated to the vertex. Denote as $\mathbf{X}^\beta$ any such allocation. In case the allocation is integer we will use $\hat{\mathbf{X}}^\beta$.
\end{defin}
 It is important to notice that the allocation rule may not be unique: In case an object lies on one of the disputing segments of $\beta$, it can be considered on both sides of the segment, and therefore it can be assigned to any of the corresponding players, or it can be split between the interested players.

Every Pareto allocation lies on the upper border of the convex set $IPS$, and a hyperplane supports $IPS$ at this point. A more precise account of the relationship between supporting hyperplanes and the Pareto allocation rule is given by the following result.
\begin{theo}
(Theorem 2 in \cite{b00})
 Assume MAC. If, for any $x =(x_1,x_2,x_3)\in \stackrel{\circ}{\Delta}_2$, we denote 
$$
RD(x) =N \left( \dfrac{1}{x_1},\dfrac{1}{x_2},\dfrac{1}{x_3} \right),
$$
then, for any $\gamma \in \stackrel{\circ}{\Delta}_2$, the allocation $\mathbf{X}^{\gamma}$ satisfies $PAR(\beta)$ with $\beta=RD(\gamma)$. Conversely, for any $\beta \in RNS$, the allocation rule $\hat{\mathbf{X}}^\beta$ supports $\mathcal{P}$ through the hyperplane $H^{\gamma}$ with $\gamma=RD(\beta)$.
\end{theo}

When, given $\beta \in RNS$, one or more objects lie on disputing segments, the associated hyperplane supports all the integer allocations and their convex hulls. This fact plays a crucial role in how the Pareto faces are generated; a detailed explanation will be provided in proof of Theorem \ref{facesTh}.

\begin{defin}
For any item $j \in K$ we denote the segments joining $a^n_j$ to each of the three vertices the {\em supporting segments}. Two items, $j, k \in K$ are {\em support independent}, or s-independent, if none of $a^n_j$ and $a^n_k$ lie on the supporting segments of the other item.
\end{defin}

In Figure \ref{fig-a} we illustrate dividing and supporting segments of one or more items.

\begin{figure}[htp]
\subfloat[\label{subfig-1:figa}]{%
\resizebox{0.32\textwidth}{!}{
\begin{tikzpicture}[scale=0.85]
\scriptsize
\centering
\draw[line width=1pt] (0,0)--(4,0);
\coordinate [label=left:$I$] (pl1) at (0,0);
\draw[line width=1pt] (0,0)--(2,3.4641);
\coordinate [label=above:$III$] (pl3) at (2,3.4641);
\draw[line width=1pt] (2,3.4641)--(4,0);
\coordinate [label=right:$II$] (pl2) at (4,0);
\filldraw[red] (2.5,1.732) circle (3pt);
\coordinate [label=below:\textcolor{red}{$a_j$}] (j) at (2.4,1.69);
\draw[line width=1pt, orange] (2.5,1.732)--(2.857,1.9795);
\draw[line width=1pt, orange] (2.5,1.732)--(3,0);
\draw[line width=1pt, orange] (2.5,1.732)--(1.6,2.77);
\draw[line width=1pt, blue] (2.5,1.732)--(0,0);
\draw[line width=1pt, blue] (2.5,1.732)--(2,3.4641);
\draw[line width=1pt, blue] (2.5,1.732)--(4,0);
\end{tikzpicture}
}
}
\hfill
\subfloat[\label{subfig-2:figa}]{%
\resizebox{0.32\textwidth}{!}{
\begin{tikzpicture}[scale=0.85]
\scriptsize
\centering
\draw[line width=1pt] (0,0)--(4,0);
\coordinate [label=left:$I$] (pl1) at (0,0);
\draw[line width=1pt] (0,0)--(2,3.4641);
\coordinate [label=above:$III$] (pl3) at (2,3.4641);
\draw[line width=1pt] (2,3.4641)--(4,0);
\coordinate [label=right:$II$] (pl2) at (4,0);
\filldraw[red] (2.5,1.732) circle (3pt);
\coordinate [label=below:\textcolor{red}{$a_j$}] (j) at (2.4,1.69);
\draw[line width=1pt, blue] (2.5,1.732)--(0,0);
\draw[line width=1pt, blue] (2.5,1.732)--(4,0);
\draw[line width=1pt, blue] (2.5,1.732)--(2,3.4641);
\filldraw[red] (1.5,0.5) circle (3pt);
\coordinate [label=below:\textcolor{red}{$a_k$}] (k) at (1.5,0.5);
\draw[line width=1pt, blue] (1.5,0.5)--(0,0);
\draw[line width=1pt, blue] (1.5,0.5)--(4,0);
\draw[line width=1pt, blue] (1.5,0.5)--(2,3.4641);
\end{tikzpicture}
}
}
\hfill
\subfloat[\label{subfig-3:figa}]{%
\resizebox{0.32\textwidth}{!}{
\begin{tikzpicture}[scale=0.85]
\scriptsize
\centering
\draw[line width=1pt] (0,0)--(4,0);
\coordinate [label=left:$I$] (pl1) at (0,0);
\draw[line width=1pt] (0,0)--(2,3.4641);
\coordinate [label=above:$III$] (pl3) at (2,3.4641);
\draw[line width=1pt] (2,3.4641)--(4,0);
\coordinate [label=right:$II$] (pl2) at (4,0);
\filldraw[red] (2.5,1.732) circle (3pt);
\coordinate [label=below:\textcolor{red}{$a_j$}] (j) at (2.4,1.69);
\draw[line width=1pt, blue] (2.5,1.732)--(0,0);
\draw[line width=1pt, blue] (2.5,1.732)--(4,0);
\draw[line width=1pt, blue] (2.5,1.732)--(2,3.4641);
\filldraw[red] (1.5,1.0392) circle (3pt);
\coordinate [label=below:\textcolor{red}{$a_k$}] (k) at (1.5,1.0392);
\draw[line width=1pt, blue] (1.5,1.0392)--(0,0);
\draw[line width=1pt, blue] (1.5,1.0392)--(4,0);
\draw[line width=1pt, blue] (1.5,1.0392)--(2,3.4641);
\end{tikzpicture}
}
}
\hfill
\caption{(a) disputing colored by orange and supporting segments colored by blue, 
(b) two s-independent items, (c) two dependent items.}
\label{fig-a}
\end{figure}
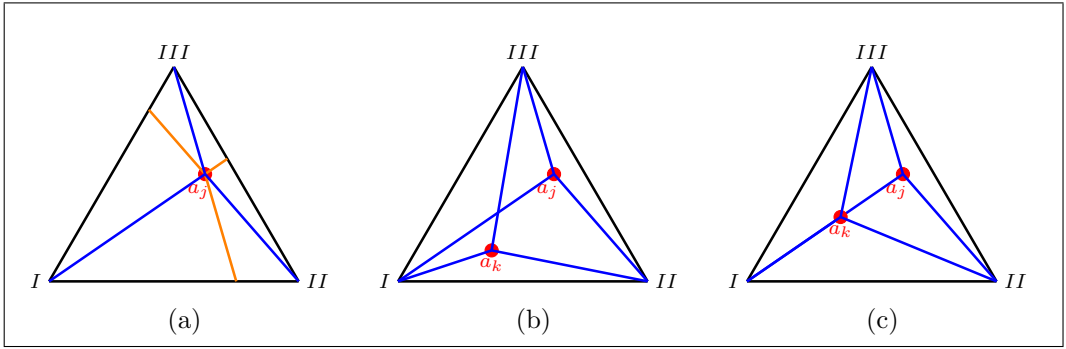

\begin{prop}
The supporting segments of two items intersect exactly once if and only if the two items are s-independent.
\end{prop}
\begin{proof}
Suppose two items, $j,h$ are $s$-independent. Then $a_j$ lies in the interior of one of the triangles that the supporting lines of $a_h$ form in the RNS, and does not lie on any of the disputing lines of the same item. The supporting line of $a_j$ with the vertex on the other side of the triangle,  will intersect a supporting line of $a_h$ once, and this is the only possible intersection.

Conversely, if two items are $s$-dependent, a supporting line of one of them will be a subsegment of the supporting line of the other, yielding an infinite number of intersections.
\end{proof}

\section{Pareto faces on the Radon-Nykodim set}
 
The following lemma consider the situation in which there are other items lying on the disputing segments between two players.
\begin{lem}  \label{edgesLemma} Under MAC, the following statements hold:
\begin{enumerate}[(i)]
\item If a disputing segment between players $i$ and $j$ of a point $\beta$ in $RNS$ contains only one item (namely, $k$),
then the hyperplane with coefficients vector $\alpha=RD(b)$ supports the partition range $IPS$ in a region containing the line segment 
$[a(\hat{\textbf{X}}^k_i), a(\hat{\textbf{X}}^k_j)]$,
where $\hat{\textbf{X}}^k_i$ ($\hat{\textbf{X}}^k_j$ respectively) denotes the allocation in which all items but $k$ are assigned according to a common PAR compatible with the coefficients and item $k$ is assigned to player $i$ (player $j$ respectively).
\item If a disputing segment between players $i$ and $j$ of a point $\beta$ in $RNS$ contains more than one item (namely, $r$ items with $r \geq 2$ and $K=\{k_1, \ldots, k_r\}$), than the hyperplane i (i) still supports the partition range $IPS$ in a region containing the line segment $[a(\hat{\textbf{X}}^{K}_i), a(	\hat{\textbf{X}}^{K}_j)]$, where $\hat{\textbf{X}}^{K}_i$ ($\hat{\textbf{X}}^{K}_j$ respectively) denotes the allocations in which all items but those in $K$ are assigned according to a common Pareto rule compatible with the coefficients, and items in $K$ are all assigned to player $i$ (player $j$, respectively). Moreover,  each point of the line segment $[a(\hat{\textbf{X}}^{K}_i), a(\hat{\textbf{X}}^{K}_j)]$ is obtained by splitting at most one item in $K$, while attributing the remaining ones in $K$ either to Player $i$ or to Player $j$ in their own entirety.
\end{enumerate}
\end{lem}
\begin{proof}
$(i)$ Let $\gamma \in \stackrel{\circ}{\Delta}_2$ be the coefficients vector of a hyperplane supporting $IPS$ and let $\beta= RD(\gamma) \in RNS$.
Following $PAR(\beta)$, two optimal integer allocations are generated, which differ only in the allocation of item $k$: $\hat{\textbf{X}}^k_i$, in which item $k$ is assigned to player $i$, 
and $\hat{\textbf{X}}^k_j$, in which item $k$ is assigned to player $j$.
So, the hyperplane characterized by $\gamma$ supports $IPS$ in $a(\hat{\textbf{X}}^k_i)$ and $a(\hat{\textbf{X}}^k_j)$ and, by Proposition 1, the whole line segment $[a(\hat{\textbf{X}}^k_i), a(\hat{\textbf{X}}^k_j)]$.

$(ii)$  Following the same line of reasoning adopted in $(i)$, the evaluations vectors $a(\hat{\textbf{X}}^{K}_i)$ and $a(\hat{\textbf{X}}^{K}_j)$ belong to the region where $IPS$ supports the given hyperplane, together with the line segment $[a(\hat{\textbf{X}}^{K}_i),a(\hat{\textbf{X}}^{K}_j)]$

We recall that goods $\ell \in K$ are characterized by the following relationship
\[
\gamma_i a_{i \ell} = \gamma_j a_{j \ell} > \gamma_h a_{h \ell} \qquad \ell \in K \; .
\]
Therefore, letting $T_i=\sum_{\ell \in K} a_{i \ell}$ and $T_j=\sum_{\ell \in K} a_{j \ell}$, we have
\[
T_j = \frac{\gamma_i T_i}{\gamma_j} \; .
\]
Moreover,
\[
a(\hat{\textbf{X}}^{K}_i)=\mathbf{c} + \mathbf{d}_i \qquad a(\hat{\textbf{X}}^{K}_j)=\mathbf{c} + \mathbf{d}_j
\]
where $\mathbf{c}$ is the evaluation vector where all the goods but those in $K$ are assigned according to a compatible common Pareto rule, $\mathbf{d}_i=(T_i,0,0)$ and $\mathbf{d}_j=(0,T_j,0)$, where, for simplicity, we assume that player $i$ (player $j$, resp.) occupies the first (second, resp.) coordinate. 

Consider now an intermediate situation where a subset of goods $H \subset K$ are assigned to player $i$ and the remaining ones in $K$ are given to $j$.  Denoting with $\hat{\mathbf{X}}_H$ the corresponding allocation, we have
\[
a(\hat{\mathbf{X}}_H) = \mathbf{c} + \mathbf{d}_H
\]
with
\begin{multline*}
\mathbf{d}_H = \left(\sum_{\ell \in H} a_{i \ell} ,T_j - \sum_{\ell \in H} a_{j \ell} ,0\right) = \left(\sum_{\ell \in H} a_{i \ell} ,\frac{\gamma_i}{\gamma_j}(T_i - \sum_{\ell \in H} a_{i \ell}) ,0\right) =
\\
\frac{\sum_{\ell \in H} a_{i \ell}}{T_i} \mathbf{d}_i +
\left(  \frac{T_i - \sum_{\ell \in H} a_{i \ell}}{T_i} \right)\mathbf{d}_j \; .
\end{multline*}
Letting $t = \frac{\sum_{\ell \in H} a_{i \ell}}{T_i}$ we therefore have
\[
a(\mathbf{x}_H) = t a(\textbf{x}^{K}_i)
+ (1 - t) a(\textbf{x}^{K}_j)
\]
and  $ a(\mathbf{x}_H) \in [a(\textbf{x}^{K}_i), a(\textbf{x}^{K}_j)]$.

To prove the last statement, consider the collection $H_p = \{k_1,\ldots,k_p\}$, $p \leq r$ and $H_0 = \varnothing$. Clearly $a(\hat{\textbf{X}}_{H_p})$, $p=0,1,\ldots,r$ spans the line segment $[a(\hat{\textbf{X}}^{K}_i), a(\hat{\textbf{X}}^{K}_j)]$, with $a(\hat{\textbf{X}}_{H_0})=a(\hat{\textbf{X}}^{K}_j)$ and $a(\hat{\textbf{X}}_{H_r}) = a(\hat{\textbf{X}}^{K}_i)$. Consequently, each point of the line segment is included between $a(\hat{\textbf{X}}_{H_{p-1}})$ and $a(\hat{\textbf{X}}_{H_{p}})$ for some $p \leq r$. Thus, if the inclusion is strict, item $p$ is split between Players $i$ and $j$, while the remaining ones in $K$ are attributed in their entirety to one player or the other.
\end{proof}

\subsection{A classification of the Pareto faces}



Faces on the Pareto surface are obtained when a given hyperplane is compatible with three or more different integer  allocations of the goods. Under MAC, this can take place only when the items are located on the disputing segments of a given hyperplane and/or coincide with the hyperplane itself, according to specific patterns listed below. 

	\begin{description}
	\item\colorbox{light-gray}{[f1]}\textbf{Faces corresponding to any $\beta \in \stackrel{\circ}{RNS}$ coinciding with  an item $a^n_j$, $j \in K$, and no other goods on the disputing segments.} In such case item $j$ can be assigned to any of the three players, while the other items are univocally assigned according to the Pareto rule. The face is a triangle, each vertex corresponding to a different assignment of  item $j$. \\
	\item\colorbox{light-gray}{[f2]}\textbf{Faces corresponding to any $\beta \in \stackrel{\circ}{RNS}$ lying at the intersection of the supporting segments of two s-independent items $a^n_j, a^n_k$ with $j,k \in K$, and no other goods in that intersection.} In this case each of the items $j$ and $k$ can be shared between two players (with only one players participating in the dispute of both items). The face is a parallelogram with each vertex corresponding to a different allocation of the pair of contested goods.\\
	\item\colorbox{light-gray}{[f3]}\textbf{Faces corresponding to any $\beta \in \stackrel{\circ}{RNS}$ lying at the intersection of the supporting segments of three s-independent items $a^n_j, a^n_k, a^n_l$, with $j,k,\ell \in K$, and no other goods in that intersection.} In such case, each of the items $j$, $k$ and $\ell$ can be shared between two players (with each player participating in two disputes out of the three). The corresponding face on the Pareto surface is a hexagon with opposite sides paralel and of equal length. Therefore 6 out of the 8 points are vertices of the hexagon, while the remaining two points lie in the interior of the face.\\
	\item\colorbox{light-gray}{[f4]}\textbf{Faces corresponding to any $\beta \in \stackrel{\circ}{RNS}$ coinciding with any item $a^n_j$, $j \in K$, and a second item $a^n_k$, $k \in K$ that lies on  a disputing segment.} The six different allocations produce a trapezoid, with the two extra points lying on its larger base.\\
	\item\colorbox{light-gray}{[f5]}\textbf{Faces corresponding to any $\beta \in \stackrel{\circ}{RNS}$ lying at the intersection of the supporting segments of two s-independent items $a^n_j, a^n_k$ with $j,k \in K$,  and an additional  item $a^n_h, h \in K$ located in the same position as  the hyperplane.} The 12 different allocations produce a face with 5 vertices and with two pairs of parallel edges with unequal length. The largest edge of each pair contains two additional points, while the remaining three points are in the interior of the face.\\
	\item\colorbox{light-gray}{[f6]}\textbf{Faces corresponding to any $\beta \in \stackrel{\circ}{RNS}$ lying at the intersection of the supporting segments of three s-independent items $a^n_j, a^n_k, a^n_{\ell}$, with $j,k,\ell \in K$,  and an additional item $a^n_h, h \in K$ located in the same position as  the hyperplane.} The 24 different distributions produce hexagons with parallel opposite sides of unequal length.
The largest edge of each pair contains two additional points, while the remaining 12 points are in the interior of the face.\\
	\end{description}

When several goods lie on a single disputing segment, these can be replaced by a single good obtained by summing up utilities for each player, and the situation can be traced back to one of the cases listed above.

\begin{theo}
\label{facesTh} Under MAC, 
each face on the Pareto Frontier is identified by one of the six characterizations listed in the above classification.
\end{theo}

\begin{proof}
A face is formed when goods are located on the disputing segments of a hyperplane in a way that three or more different goods' allocations generate unaligned points on the Pareto surface of the partition range.
Moreover, according to Lemma \ref{edgesLemma}, when several goods lie on a disputing segment (given a fixed common Pareto rule for the remaining goods) an edge on the Pareto surface is produced and it is equivalent to the one that we obtain replacing all those disputed goods with a single one by summing up utilities for each player.

Based on these simple remarks, we notice that different faces are formed depending on whether (a) the hyperplane coincides with a good in $\stackrel{\circ}{RNS}$ and (b) other items lie on one, two or all the disputing segments originated by the same hyperplane. Therefore, six cases [{\bf (f1)} through {\bf (f6)}] are generated. Let us analyze each case with the aid of Figure \ref{face_f1} through Figure \ref{face_f6}.

In each of them, we move  $\beta$ that determines the PAR out of its original location by a small step in the directions shown by the arrows. This corresponds to a slight tilt of the supporting hyperplane on the Pareto surface, so that the single edge becomes the only supporting region for the hyperplane. 

We obtain edges that are connected to each other and form a cycle.
 Moreover, every time we consider the hyperplane moving in opposite directions in the RNS diagram  (actually all cases but {\bf (f1)}) we consider goods that are contested between players with constant utility ratio. This yields parallel edges of the faces. We distinguish between two cases:

\begin{itemize}

\item In cases {\bf (f2)} and {\bf (f3)}, when moving the hyperplane in opposite directions, the contested good remains the same. In such cases the opposite sides are not only parallel, but also of equal length\footnote{As an example, let us analyze the case \textbf{(f2)} in Figure \ref{face_f2}. Case \textbf{(f2)} corresponds to the situation in which any $b \in \triangle_2$ lying at the intersection of the supporting segments of two s-independent items $a_j, a_k$ with $j,k \in K$, and no other goods in that intersection. Looking at the geometrical representation in the simplex, we are able to identify four directions (shown by black arrows and labeled by A, B, C and D) that correspond to a slight tilt of the supporting hyperplane on the Pareto surface, so that the single edge becomes the only supporting region for the hyperplane. They are opposite two by two; namely, A is opposite to B and C is opposite to D. Moving towards A (see Figure \ref{subfig-2:f2}), the initial  allocation changes as follow: item 1 (in blue) ends to be disputed and is assigned to player III; item 2 (in red) is still disputed between players I and II. Two possible assignments are generated: III, I and III, II; they are vertices of edge A. Moving towards the opposite direction B, the initial allocation changes as follow: item 1 ends to be disputed and is assigned to player I; item 2 is still disputed between players I and II. Two possible assignments are generated: I,I and I, II; they are vertices of edge B. The contested good remains the same in both cases; i.e., item 2. Applying an analogous reasoning, moving towards direction C and D item 2 ends to be disputed and is assigned to player II \textst{or} (player I, respectively); item 1 is still contended between players  I and III. Still, the contested good remains the same; i.e., item 1. Hence, we can conclude that edges generated by A and B are not only parallels (since moving in opposite directions we have constant utility ratio across players) but also of equal length (since the contested good remains the same). Same reasoning applies to edges generated by C and D.}.

\item In cases {\bf (f4)}, {\bf (f5)} and {\bf (f6)}, the number of contested goods changes from 1 to 2  when the opposite direction is taken.  Correspondingly, the length of the side changes, the larger side corresponding to the case with two contested goods. The interior points on this larger side correspond to the  intermediate cases where one good is allotted to each player\footnote{As an example, let us analyze the case \textbf{(f4)} in Figure \ref{face_f4}. Case \textbf{(f4)} corresponds to the situation in which any $b \in \triangle_2$ coinciding with any item $a_j$, $j \in K$, and a second item $a_k$, $k \in K$ that lies on a disputing segment. Looking at the geometrical representation in the simplex, we identify four directions (shown by black arrows and labeled by A, B, C and D) that correspond to a slight tilt of the supporting hyperplane on the Pareto surface, so that the single edge becomes the only supporting region for the hyperplane.  Only two of them are opposite directions; namely, C and D. Moving towards C (see Figure \ref{subfig-2:f4}) items 1 and 2 are still contested between players I and II; player III exits the dispute. Four possible assignments are generated: I, II - I, I - II, I and II, II. Two of them (namely, I, I and II, II) are vertices of edge C. The remaining two are points that lie on the same edge; their respective position depends on the matrix of evaluation. Moving towards D (see Figure \ref{subfig-3:f4}) item 1 ends to be disputed and it is assigned to player III; at the same time, item 2 is still contested between player I and player II. Hence, two possible assignments are generated: III, I and III, II. They are vertices of edge D. The number of contested goods is different in the two analyzed directions.; namely, it is equal to two with respect to direction C and equal to one with respect to direction D. As a result, edges generated by C and D are parallel (since moving in opposite directions we have constant utility ratio across players) but length of C is larger than length of D since in C the number of contested goods is greater. The interior points on this larger side correspond to the  intermediate cases where one good is allotted to each player.}.

\end{itemize}

\end{proof}

\begin{figure}[!htb]
    \subfloat[\textbf{(f1)} - RN Set\label{subfig-1:f1}]{%
        \resizebox{0.48\textwidth}{!}{
	\begin{tikzpicture}
	\centering
	\draw[line width=1pt] (0,0)--(4,0);
	\coordinate [label=left:$I$] (pl1) at (0,0);
	\draw[line width=1pt] (0,0)--(2,3.4641);
	\coordinate [label=above:$III$] (pl3) at (2,3.4641);
	\draw[line width=1pt] (2,3.4641)--(4,0);
	\coordinate [label=right:$II$] (pl2) at (4,0);
	\filldraw[blue] (2.084,1.2) circle (2pt);
	\draw[orange] (2.084,1.2) circle (3.5pt);
	\draw[line width=0.1pt, lightgray] (0,0)--(3,1.67);
	\draw[line width=0.1pt, lightgray] (2,3.4641)--(2.1253,0);
	\draw[line width=0.1pt, lightgray] (4,0)--(1.03,1.7825);
	\draw[line width=1pt, orange] (2.084,1.15)--(3,1.67);
	\draw[line width=1pt, orange] (2.084,1.15)--(2.1253,0);
	\draw[line width=1pt, orange] (2.084,1.15)--(1.03,1.7825);
	\draw[->] (2.0824,1.15) -- (2.0662,1.65) node [above] {\tiny C};
	\draw[->] (2.0824,1.15) -- (2.5,0.9) node [below] {\tiny B};
	\draw[->] (2.0824,1.15) -- (1.65,0.9) node [below] {\tiny A};
	\begin{customlegend}[
	legend entries={good 1 (\textcolor{blue}{$a_j$}), hyperplane (\textcolor{orange}{$b$})},
	legend style={at={(1,4)}, font=\tiny}
	]
    	\addlegendimage{blue, only marks}
    	\addlegendimage{mark=o, only marks, color=orange}
    	\end{customlegend}
	\end{tikzpicture}
	}
    }
    \hfill
    \subfloat[Pareto Surface of $\mathcal{P}$\label{subfig-2:f1}]{%
        \resizebox{0.48\textwidth}{!}{
	\begin{tikzpicture}
	\centering
	\draw[line width=1pt] (0,0)--(4,0);
	\coordinate [label=left:\text{(1,0,0)}] (pl1) at (0,0);
	\draw[line width=1pt] (0,0)--(2,3.4641);
	\coordinate [label=above:\text{(0,0,1)}] (pl3) at (2,3.4641);
	\draw[line width=1pt] (2,3.4641)--(4,0);
	\coordinate [label=right:\text{(0,1,0)}] (pl2) at (4,0);
	\draw [black, fill=light-gray] (0,0) -- (4,0) -- (2,3.4641) -- (0,0);
	\filldraw[black] (2.3,2) circle (2pt);
	\coordinate [label=above:\tiny $III$ ] (vertex1) at (2.3,2);
	\filldraw[black] (1.4,0.75) circle (2pt);
	\coordinate [label=below:\tiny $I$ ] (vertex2) at (1.4,0.75);
	\filldraw[black] (3,0.45) circle (2pt);
	\coordinate [label=below:\tiny $II$ ] (vertex3) at (3,0.45);
	\draw[line width=1pt, black] (2.3,2) -- (3,0.45) node [midway, right] {\tiny A};
	\draw[line width=1pt, black] (1.4,0.75) -- (2.3,2) node [midway, left] {\tiny B};
	\draw[line width=1pt, black] (1.4,0.75) -- (3,0.45) node [midway, below] {\tiny C};	
	\end{tikzpicture}
	}
    }
\caption{Simplex and Pareto surface of $\mathcal{P}$ for face {\bf (f1)}.\\ Item $j$ can be assigned to any of the three players, while the other items are univocally assigned according to the Pareto rule. Hence, three possible allocations for the disputed good are generated. The face is a \textbf{triangle}, each vertex corresponding to a different assignment of item $j$.}
\label{face_f1}
\end{figure}
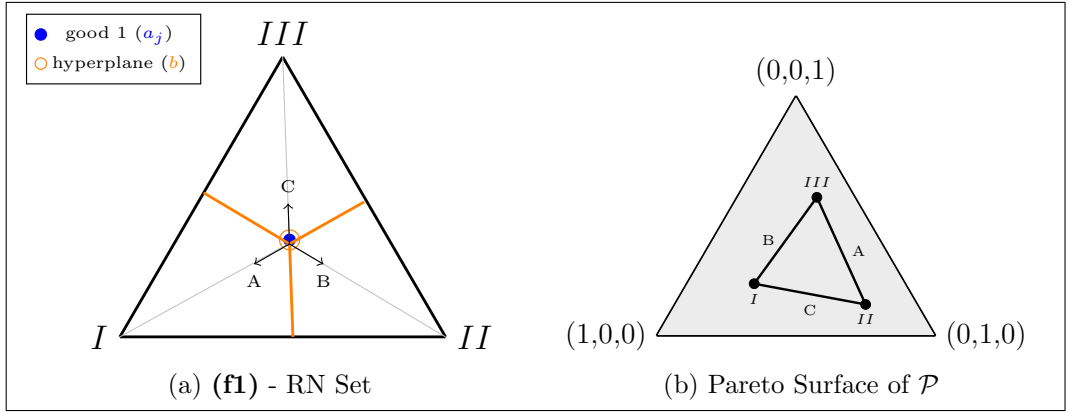
 
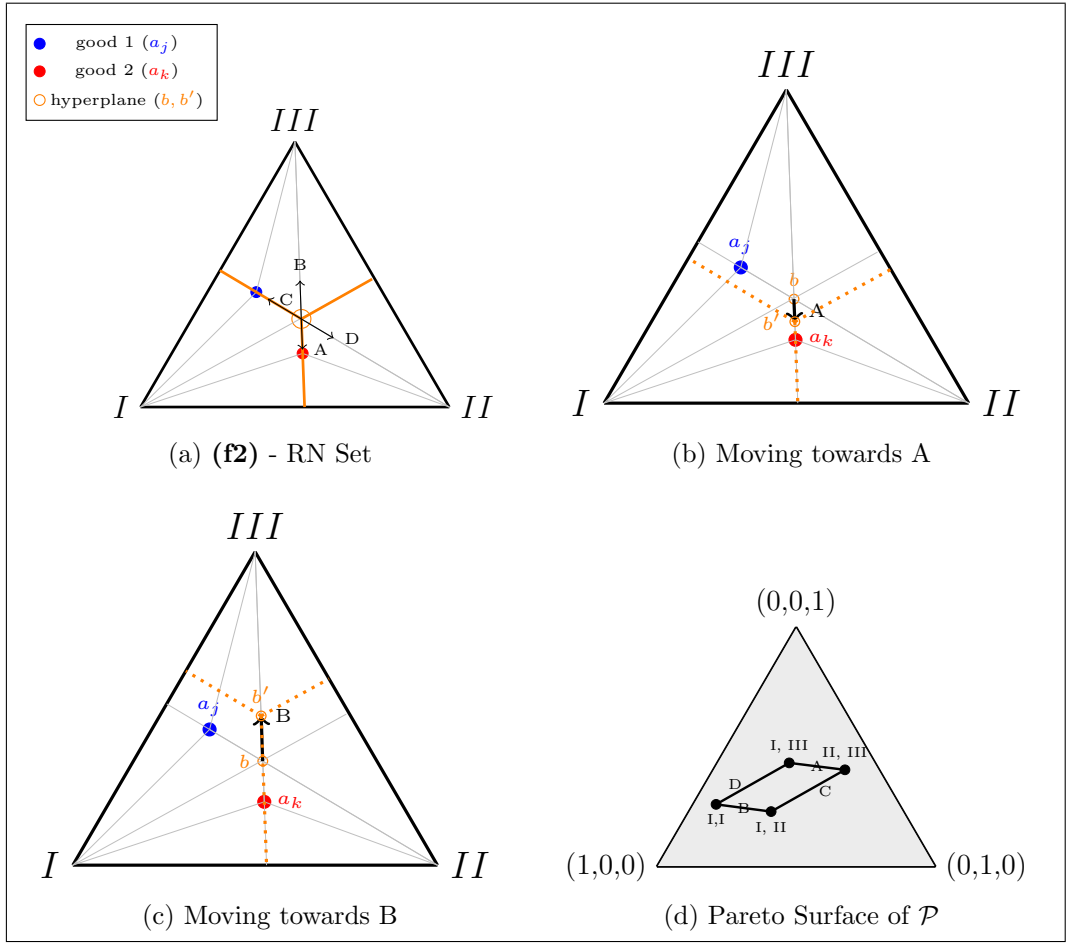
\begin{figure}[!htb]
    \subfloat[\textbf{(f2)} - RN Set\label{subfig-1:f2}]{%
        \resizebox{0.48\textwidth}{!}{
	\begin{tikzpicture}
	\centering
	\draw[line width=1pt] (0,0)--(4,0);
	\coordinate [label=left:$I$] (pl1) at (0,0);
	\draw[line width=1pt] (0,0)--(2,3.4641);
	\coordinate [label=above:$III$] (pl3) at (2,3.4641);
	\draw[line width=1pt] (2,3.4641)--(4,0);
	\coordinate [label=right:$II$] (pl2) at (4,0);
	\filldraw[blue] (1.5,1.5) circle (2pt);
	\filldraw[red] (2.1,0.7) circle (2pt);
	\draw[line width=0.1pt, lightgray] (0,0)--(1.5,1.5);
	\draw[line width=0.1pt, lightgray] (2,3.4641)--(1.5,1.5);
	\draw[line width=0.1pt, lightgray] (4,0)--(1.5,1.5);
	\draw[line width=0.1pt, lightgray] (0,0)--(2.1,0.7);
	\draw[line width=0.1pt, lightgray] (2,3.4641)--(2.1,0.7);
	\draw[line width=0.1pt, lightgray] (4,0)--(2.1,0.7);
	\draw[orange] (2.084,1.15) circle (3.5pt);
	\draw[line width=0.1pt, lightgray] (0,0)--(3,1.67);
	\draw[line width=0.1pt, lightgray] (2,3.4641)--(2.1253,0);
	\draw[line width=0.1pt, lightgray] (4,0)--(1.03,1.7825);
	\draw[line width=1pt, orange] (2.084,1.15)--(3,1.67);
	\draw[line width=1pt, orange] (2.084,1.15)--(2.1253,0);
	\draw[line width=1pt, orange] (2.084,1.15)--(1.03,1.7825);
	\draw[->] (2.0824,1.15) -- (2.0662,1.65) node [above] {\tiny B};
	\draw[->] (2.0824,1.15) -- (2.5,0.9) node [below, right] {\tiny D};
	\draw[->] (2.0824,1.15) -- (1.65,1.41) node [above, right] {\tiny C};
	\draw[->] (2.0824,1.15) -- (2.0982,0.75) node [below, right] {\tiny A};
	\begin{customlegend}[
	legend entries={good 1 (\textcolor{blue}{$a_j$}), good 2 (\textcolor{red}{$a_k$}), hyperplane (\textcolor{orange}{$b, b'$})},
	legend style={at={(1,5)}, font=\tiny}
	]
    	\addlegendimage{blue, only marks}
    	\addlegendimage{red, only marks}
    	\addlegendimage{mark=o, only marks, color=orange}
    	\end{customlegend}
	\end{tikzpicture}
	}
    }
    \hfill
    \subfloat[Moving towards A\label{subfig-2:f2}]{%
        \resizebox{0.48\textwidth}{!}{
	\begin{tikzpicture}
	\centering
	\draw[line width=1pt] (0,0)--(4,0);
	\coordinate [label=left:$I$] (pl1) at (0,0);
	\draw[line width=1pt] (0,0)--(2,3.4641);
	\coordinate [label=above:$III$] (pl3) at (2,3.4641);
	\draw[line width=1pt] (2,3.4641)--(4,0);
	\coordinate [label=right:$II$] (pl2) at (4,0);
	\filldraw[blue] (1.5,1.5) circle (2pt);
	\coordinate [label=above:\tiny \textcolor{blue}{$a_j$}] (j) at (1.5,1.5);
	\filldraw[red] (2.1,0.7) circle (2pt);
	\coordinate [label=right:\tiny \textcolor{red}{$a_k$}] (k) at (2.1,0.7);
	\draw[line width=0.1pt, lightgray] (0,0)--(1.5,1.5);
	\draw[line width=0.1pt, lightgray] (2,3.4641)--(1.5,1.5);
	\draw[line width=0.1pt, lightgray] (4,0)--(1.5,1.5);
	\draw[line width=0.1pt, lightgray] (0,0)--(2.1,0.7);
	\draw[line width=0.1pt, lightgray] (2,3.4641)--(2.1,0.7);
	\draw[line width=0.1pt, lightgray] (4,0)--(2.1,0.7);
	\draw[orange] (2.084,1.15) circle (1.5pt);
	\coordinate [label=above:\tiny \textcolor{orange}{$b$}] (b) at (2.084,1.15);
	\draw[line width=0.1pt, lightgray] (0,0)--(3,1.67);
	\draw[line width=0.1pt, lightgray] (2,3.4641)--(2.1253,0);
	\draw[line width=0.1pt, lightgray] (4,0)--(1.03,1.7825);
	\draw[->, line width=1pt] (2.0824,1.15) -- (2.093,0.9) node [midway,right] {\tiny A};
	\draw[orange] (2.093,0.9) circle (1.5pt);
	\coordinate [label=left:\tiny \textcolor{orange}{$b'$}] (nuovo_b) at (2.093,0.9);
	\draw[line width=1pt, orange, dotted] (2.093,0.9)--(3.146,1.479);
	\draw[line width=1pt, orange, dotted] (2.093,0.9)--(2.1253,0);
	\draw[line width=1pt, orange, dotted] (2.093,0.9)--(0.9244,1.6);
	\end{tikzpicture}
	}
    }
    \hfill
    \subfloat[Moving towards B\label{subfig-3:f2}]{%
        \resizebox{0.48\textwidth}{!}{
	\begin{tikzpicture}
	\centering
	\draw[line width=1pt] (0,0)--(4,0);
	\coordinate [label=left:$I$] (pl1) at (0,0);
	\draw[line width=1pt] (0,0)--(2,3.4641);
	\coordinate [label=above:$III$] (pl3) at (2,3.4641);
	\draw[line width=1pt] (2,3.4641)--(4,0);
	\coordinate [label=right:$II$] (pl2) at (4,0);
	\filldraw[blue] (1.5,1.5) circle (2pt);
	\coordinate [label=above:\tiny \textcolor{blue}{$a_j$}] (j) at (1.5,1.5);
	\filldraw[red] (2.1,0.7) circle (2pt);
	\coordinate [label=right:\tiny \textcolor{red}{$a_k$}] (k) at (2.1,0.7);
	\draw[line width=0.1pt, lightgray] (0,0)--(1.5,1.5);
	\draw[line width=0.1pt, lightgray] (2,3.4641)--(1.5,1.5);
	\draw[line width=0.1pt, lightgray] (4,0)--(1.5,1.5);
	\draw[line width=0.1pt, lightgray] (0,0)--(2.1,0.7);
	\draw[line width=0.1pt, lightgray] (2,3.4641)--(2.1,0.7);
	\draw[line width=0.1pt, lightgray] (4,0)--(2.1,0.7);
	\draw[orange] (2.084,1.15) circle (1.5pt);
	\coordinate [label=left:\tiny \textcolor{orange}{$b$}] (b) at (2.084,1.15);
	\draw[line width=0.1pt, lightgray] (0,0)--(3,1.67);
	\draw[line width=0.1pt, lightgray] (2,3.4641)--(2.1253,0);
	\draw[line width=0.1pt, lightgray] (4,0)--(1.03,1.7825);
	\draw[->, line width=1pt] (2.0824,1.15) -- (2.0662,1.65) node [right] {\tiny B};
	\draw[orange] (2.0662,1.65) circle (1.5pt);
	\coordinate [label=above:\tiny \textcolor{orange}{$b'$}] (nuovo_b) at (2.0662,1.65);
	\draw[line width=1pt, orange, dotted] (2.0662,1.65)--(2.81,2.06);
	\draw[line width=1pt, orange, dotted] (2.0662,1.65)--(2.1253,0);
	\draw[line width=1pt, orange, dotted] (2.0662,1.65)--(1.24,2.1465);
	\end{tikzpicture}
	}
    }
    \hfill
    \subfloat[Pareto Surface of $\mathcal{P}$\label{subfig-4:f2}]{%
        \resizebox{0.48\textwidth}{!}{
	\begin{tikzpicture}
	\centering
	\draw[line width=1pt] (0,0)--(4,0);
	\coordinate [label=left:\text{(1,0,0)}] (pl1) at (0,0);
	\draw[line width=1pt] (0,0)--(2,3.4641);
	\coordinate [label=above:\text{(0,0,1)}] (pl3) at (2,3.4641);
	\draw[line width=1pt] (2,3.4641)--(4,0);
	\coordinate [label=right:\text{(0,1,0)}] (pl2) at (4,0);
	\draw [black, fill=light-gray] (0,0) -- (4,0) -- (2,3.4641) -- (0,0);
	\filldraw[black] (1.9,1.5) circle (2pt);
	\coordinate [label=above:\tiny \text{I, III} ] (vertex1) at (1.9,1.5);
	\filldraw[black] (2.7,1.4) circle (2pt);
	\coordinate [label=above:\tiny \text{II, III} ] (vertex2) at (2.7,1.4);
	\filldraw[black] (0.85,0.9) circle (2pt);
	\coordinate [label=below:\tiny \text{I,I} ] (vertex3) at (0.85,0.9);
	\filldraw[black] (1.64,0.7955) circle (2pt);
	\coordinate [label=below:\tiny \text{I, II} ] (vertex4) at (1.64,0.7955);
	\draw[line width=1pt, black] (1.9,1.5) -- (2.7,1.4) node [midway] {\tiny A};
	\draw[line width=1pt, black] (1.64,0.7955) -- (0.85,0.9) node [midway] {\tiny B};
	\draw[line width=1pt, black] (2.7,1.4) -- (1.64,0.7955) node [midway, right] {\tiny C};
	\draw[line width=1pt, black] (1.9,1.5) -- (0.85,0.9) node [midway, left] {\tiny D};
	\end{tikzpicture}
	}
    }  
\caption{Simplex and Pareto surface of $\mathcal{P}$ for face {\bf (f2)}.\\ Both item $j$ and item $k$ can be shared between two players (with only one players participating in the dispute of both items). The face is a \textbf{parallelogram} with each vertex corresponding to a different allocation of the two contested items $j$ and $k$.}
\label{face_f2}
\end{figure}

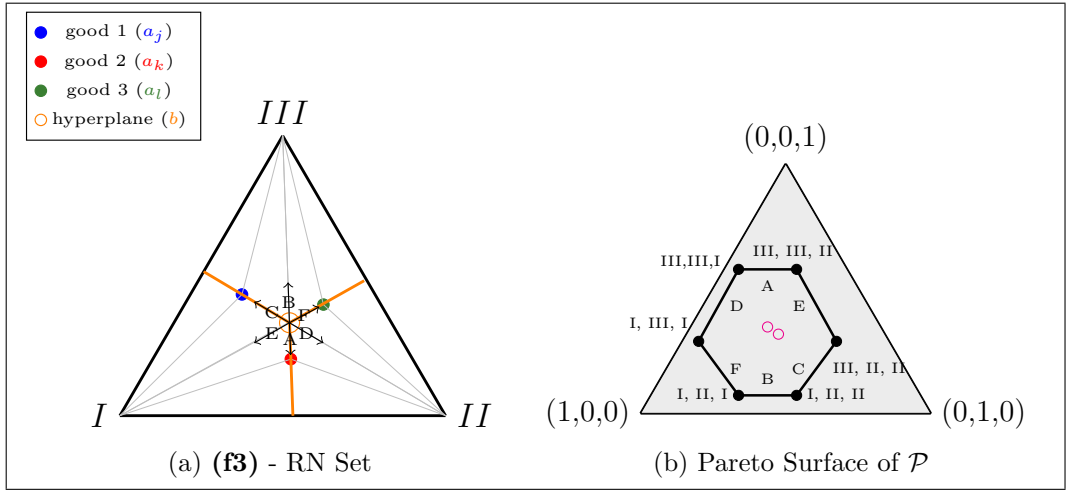
\begin{figure}[!htb]
    \subfloat[\textbf{(f3)} - RN Set\label{subfig-1:f3}]{%
        \resizebox{0.48\textwidth}{!}{
	\begin{tikzpicture}
	\centering
	\draw[line width=1pt] (0,0)--(4,0);
	\coordinate [label=left:$I$] (pl1) at (0,0);
	\draw[line width=1pt] (0,0)--(2,3.4641);
	\coordinate [label=above:$III$] (pl3) at (2,3.4641);
	\draw[line width=1pt] (2,3.4641)--(4,0);
	\coordinate [label=right:$II$] (pl2) at (4,0);
	\filldraw[blue] (1.5,1.5) circle (2pt);
	\filldraw[red] (2.1,0.7) circle (2pt);
	\filldraw[OliveGreen] (2.5,1.375) circle (2pt);
	\draw[line width=0.1pt, lightgray] (0,0)--(1.5,1.5);
	\draw[line width=0.1pt, lightgray] (2,3.4641)--(1.5,1.5);
	\draw[line width=0.1pt, lightgray] (4,0)--(1.5,1.5);
	\draw[line width=0.1pt, lightgray] (0,0)--(2.1,0.7);
	\draw[line width=0.1pt, lightgray] (2,3.4641)--(2.1,0.7);
	\draw[line width=0.1pt, lightgray] (4,0)--(2.1,0.7);
	\draw[line width=0.1pt, lightgray] (0,0)--(2.5,1.375);
	\draw[line width=0.1pt, lightgray] (2,3.4641)--(2.5,1.375);
	\draw[line width=0.1pt, lightgray] (4,0)--(2.5,1.375);
	\draw[orange] (2.084,1.15) circle (3.5pt);
	\draw[line width=0.1pt, lightgray] (0,0)--(3,1.67);
	\draw[line width=0.1pt, lightgray] (2,3.4641)--(2.1253,0);
	\draw[line width=0.1pt, lightgray] (4,0)--(1.03,1.7825);
	\draw[line width=1pt, orange] (2.084,1.15)--(3,1.67);
	\draw[line width=1pt, orange] (2.084,1.15)--(2.1253,0);
	\draw[line width=1pt, orange] (2.084,1.15)--(1.03,1.7825);
	\draw[->] (2.0824,1.15) -- (2.0982,0.75) node [midway] {\tiny A};
	\draw[->] (2.0824,1.15) -- (2.0662,1.65) node [midway] {\tiny B};
	\draw[->] (2.0824,1.15) -- (1.65,1.41) node [midway] {\tiny C};
	\draw[->] (2.0824,1.15) -- (2.5,0.9) node [midway] {\tiny D};
	\draw[->] (2.0824,1.15) -- (1.65,0.9) node [midway] {\tiny E};
	\draw[->] (2.0824,1.15) -- (2.45,1.3475) node [midway] {\tiny F};
	\begin{customlegend}[
	legend entries={good 1 (\textcolor{blue}{$a_j$}), good 2 (\textcolor{red}{$a_k$}), good 3 (\textcolor{OliveGreen}{$a_l$}), hyperplane (\textcolor{orange}{$b$})},
	legend style={at={(1,5)}, font=\tiny}
	]
    	\addlegendimage{blue, only marks}
    	\addlegendimage{red, only marks}
    	\addlegendimage{OliveGreen, only marks}
	\addlegendimage{mark=o, only marks, color=orange}
    	\end{customlegend}
	\end{tikzpicture}
	}
    }
    \hfill
    \subfloat[Pareto Surface of $\mathcal{P}$\label{subfig-2:f3}]{%
        \resizebox{0.5\textwidth}{!}{
	\begin{tikzpicture}
	\centering
	\draw[line width=1pt] (0,0)--(4,0);
	\coordinate [label=left:\text{(1,0,0)}] (pl1) at (0,0);
	\draw[line width=1pt] (0,0)--(2,3.4641);
	\coordinate [label=above:\text{(0,0,1)}] (pl3) at (2,3.4641);
	\draw[line width=1pt] (2,3.4641)--(4,0);
	\coordinate [label=right:\text{(0,1,0)}] (pl2) at (4,0);
	\draw [black, fill=light-gray] (0,0) -- (4,0) -- (2,3.4641) -- (0,0);
	\filldraw[black] (1.35,2) circle (2pt);
	\coordinate [label=left:\tiny \text{III,III,I} ] (vertex1) at (1.25,2.1);
	\filldraw[black] (0.8,1) circle (2pt);
	\coordinate [label=above:\tiny \text{I, III, I} ] (vertex2) at (0.25,1);
	\filldraw[black] (1.35,0.25) circle (2pt);
	\coordinate [label=left:\tiny \text{I, II, I} ] (vertex3) at (1.35,0.25);
	\filldraw[black] (2.15,0.25) circle (2pt);
	\coordinate [label=right:\tiny \text{I, II, II} ] (vertex4) at (2.15,0.25);
	\filldraw[black] (2.70,1) circle (2pt);
	\coordinate [label=below:\tiny \text{III, II, II} ] (vertex5) at (3.15,0.85);
	\filldraw[black] (2.15,2) circle (2pt);
	\coordinate [label=above:\tiny \text{III, III, II} ] (vertex6) at (2.1,2);
	\draw[Magenta] (1.9,1.1) circle (2pt);
	\draw[Magenta] (1.75,1.2) circle (2pt);
	\draw[line width=1pt, black] (1.35,2) -- (0.8,1) node [midway, right] {\tiny D};
	\draw[line width=1pt, black] (0.8,1) -- (1.35,0.25) node [midway, right] {\tiny F};
	\draw[line width=1pt, black] (1.35,0.25) -- (2.15,0.25) node [midway, above] {\tiny B};
	\draw[line width=1pt, black] (2.15,0.25) -- (2.70,1) node [midway, left] {\tiny C};
	\draw[line width=1pt, black] (2.70,1) -- (2.15,2) node [midway, left] {\tiny E};
	\draw[line width=1pt, black] (2.15,2) -- (1.35,2) node [midway, below] {\tiny A};
	\end{tikzpicture}
	}
    }
\caption{Simplex and Pareto surface of $\mathcal{P}$ for face {\bf (f3)}.\\ Each of the items $j$, $k$ and $l$ can be shared between two players (with each player participating in two disputes out of the three). Eight possible assignments are generated and (consequently) eight points in the simplex; six of them are vertices of the \textbf{hexagon}, while the remaining two points lie in the interior of the face.}
\label{face_f3}
\end{figure}

\begin{figure}[!htb]
    \subfloat[\textbf{(f4)} - RN Set\label{subfig-1:f4}]{%
        \resizebox{0.48\textwidth}{!}{
	\begin{tikzpicture}
	\centering
	\draw[line width=1pt] (0,0)--(4,0);
	\coordinate [label=left:$I$] (pl1) at (0,0);
	\draw[line width=1pt] (0,0)--(2,3.4641);
	\coordinate [label=above:$III$] (pl3) at (2,3.4641);
	\draw[line width=1pt] (2,3.4641)--(4,0);
	\coordinate [label=right:$II$] (pl2) at (4,0);
	\filldraw[blue] (2.084,1.15) circle (2pt);
	\draw[orange] (2.084,1.15) circle (3.5pt);
	\filldraw[red] (2.1,0.7) circle (2pt);
	\draw[line width=0.1pt, lightgray] (0,0)--(2.1,0.7);
	\draw[line width=0.1pt, lightgray] (2,3.4641)--(2.1,0.7);
	\draw[line width=0.1pt, lightgray] (4,0)--(2.1,0.7);
	\draw[line width=0.1pt, lightgray] (0,0)--(3,1.67);
	\draw[line width=0.1pt, lightgray] (2,3.4641)--(2.1253,0);
	\draw[line width=0.1pt, lightgray] (4,0)--(1.03,1.7825);
	\draw[line width=1pt, orange] (2.084,1.15)--(3,1.67);
	\draw[line width=1pt, orange] (2.084,1.15)--(2.1253,0);
	\draw[line width=1pt, orange] (2.084,1.15)--(1.03,1.7825);
	\draw[->] (2.0824,1.15) -- (1.65, 0.9) node [below, left] {\tiny A};
	\draw[->] (2.0824,1.15) -- (2.5,0.9) node [below, right] {\tiny B};
	\draw[->] (2.0824,1.15) -- (2.0662,1.65) node [above] {\tiny C};
	\draw[->] (2.0824,1.15) -- (2.0982,0.75) node [below, left] {\tiny D};
	\begin{customlegend}[
	legend entries={good 1 (\textcolor{blue}{$a_j$}), good 2 (\textcolor{red}{$a_k$}), hyperplane (\textcolor{orange}{$b, b'$})},
	legend style={at={(1,5)}, font=\tiny}
	]
    	\addlegendimage{blue, only marks}
    	\addlegendimage{red, only marks}
	\addlegendimage{mark=o, only marks, color=orange}
    	\end{customlegend}
	\end{tikzpicture}
	}
    }
    \hfill
        \subfloat[Moving towards C\label{subfig-2:f4}]{%
        \resizebox{0.48\textwidth}{!}{
	\begin{tikzpicture}
	\centering
	\draw[line width=1pt] (0,0)--(4,0);
	\coordinate [label=left:$I$] (pl1) at (0,0);
	\draw[line width=1pt] (0,0)--(2,3.4641);
	\coordinate [label=above:$III$] (pl3) at (2,3.4641);
	\draw[line width=1pt] (2,3.4641)--(4,0);
	\coordinate [label=right:$II$] (pl2) at (4,0);
	\filldraw[blue] (2.084,1.15) circle (2pt);
	\draw[orange] (2.084,1.15) circle (3.5pt);
	\coordinate [label=right:\tiny \textcolor{orange}{b} \textcolor{blue}{$\equiv a_j$} ] (good1) at (2.084,1.13);
	\filldraw[red] (2.1,0.7) circle (2pt);
	\coordinate [label=left:\tiny \textcolor{red}{$a_k$}] (good2) at (2.1,0.8);
	\draw[line width=0.1pt, lightgray] (0,0)--(2.1,0.7);
	\draw[line width=0.1pt, lightgray] (2,3.4641)--(2.1,0.7);
	\draw[line width=0.1pt, lightgray] (4,0)--(2.1,0.7);
	\draw[line width=0.1pt, lightgray] (0,0)--(3,1.67);
	\draw[line width=0.1pt, lightgray] (2,3.4641)--(2.1253,0);
	\draw[line width=0.1pt, lightgray] (4,0)--(1.03,1.7825);
	\draw[->] (2.0824,1.15) -- (2.0662,1.65) node [right] {\tiny C};
	\draw[orange] (2.061,1.8) circle (3.5pt);
	\coordinate [label=above:\tiny \textcolor{orange}{b'}] (b') at (2.061,1.8);
	\draw[line width=1pt, orange, dotted] (2.061,1.8)--(2.1253,0);
	\draw[line width=1pt, orange, dotted] (2.061,1.8)--(1.3,2.255);
	\draw[line width=1pt, orange, dotted] (2.061,1.8)--(2.744,2.1753);
	\end{tikzpicture}
	}
    }
    \hfill
    \subfloat[Moving towards D\label{subfig-3:f4}]{%
        \resizebox{0.48\textwidth}{!}{
	\begin{tikzpicture}
	\centering
	\draw[line width=1pt] (0,0)--(4,0);
	\coordinate [label=left:$I$] (pl1) at (0,0);
	\draw[line width=1pt] (0,0)--(2,3.4641);
	\coordinate [label=above:$III$] (pl3) at (2,3.4641);
	\draw[line width=1pt] (2,3.4641)--(4,0);
	\coordinate [label=right:$II$] (pl2) at (4,0);
	\filldraw[blue] (2.084,1.15) circle (2pt);
	\draw[orange] (2.084,1.15) circle (3.5pt);
	\coordinate [label=above:\tiny \textcolor{orange}{b} \textcolor{blue}{$\equiv a_j$} ] (good1) at (2.084,1.13);
	\filldraw[red] (2.1,0.7) circle (2pt);
	\coordinate [label=below:\tiny \textcolor{red}{$a_k$}] (good2) at (2.1,0.7);
	\draw[line width=0.1pt, lightgray] (0,0)--(2.1,0.7);
	\draw[line width=0.1pt, lightgray] (2,3.4641)--(2.1,0.7);
	\draw[line width=0.1pt, lightgray] (4,0)--(2.1,0.7);
	\draw[line width=0.1pt, lightgray] (0,0)--(3,1.67);
	\draw[line width=0.1pt, lightgray] (2,3.4641)--(2.1253,0);
	\draw[line width=0.1pt, lightgray] (4,0)--(1.03,1.7825);
	\draw[->] (2.0824,1.15) -- (2.093,0.9) node [below, right] {\tiny D};
	\draw[orange] (2.093,0.9) circle (1.5pt);
	\coordinate [label=left:\tiny \textcolor{orange}{b'}] (b') at (2.093,0.9);
	\draw[line width=1pt, orange, dotted] (2.093,0.9)--(2.1253,0);
	\draw[line width=1pt, orange, dotted] (2.093,0.9)--(0.924,1.6);
	\draw[line width=1pt, orange, dotted] (2.093,0.9)--(3.146,1.48);
	\end{tikzpicture}
	}
    }
    \hfill
    \subfloat[Pareto Surface of $\mathcal{P}$\label{subfig-4:f4}]{%
        \resizebox{0.48\textwidth}{!}{
	\begin{tikzpicture}
	\centering
	\draw[line width=1pt] (0,0)--(4,0);
	\coordinate [label=left:\text{(1,0,0)}] (pl1) at (0,0);
	\draw[line width=1pt] (0,0)--(2,3.4641);
	\coordinate [label=above:\text{(0,0,1)}] (pl3) at (2,3.4641);
	\draw[line width=1pt] (2,3.4641)--(4,0);
	\coordinate [label=right:\text{(0,1,0)}] (pl2) at (4,0);
	\draw [black, fill=light-gray] (0,0) -- (4,0) -- (2,3.4641) -- (0,0);
	\filldraw[black] (1.35,2) circle (2pt);
	\coordinate [label=left:\tiny \text{III,I} ] (vertex1) at (1.25,2.2);
	\filldraw[black] (2.15,2) circle (2pt);
	\coordinate [label=above:\tiny \text{III, II} ] (vertex2) at (2.15,2);
	\filldraw[black] (0.9,0.6) circle (2pt);
	\coordinate [label=below:\tiny \text{I,I} ] (vertex3) at (0.9,0.6);
	\filldraw[black] (3.2,0.6) circle (2pt);
	\coordinate [label=below:\tiny \text{II,II} ] (vertex4) at (3.2,0.6);
	\draw[Magenta] (1.4, 0.6) circle (2pt);
	\draw[Magenta] (2.08, 0.6) circle (2pt);
	\draw[line width=1pt, black] (1.35,2) -- (2.15,2) node [midway, below] {\tiny D};
	\draw[line width=1pt, black] (2.15,2) -- (3.2,0.6) node [midway, right] {\tiny A};
	\draw[line width=1pt, black] (3.2,0.6) -- (0.9,0.6) node [midway, below] {\tiny C};
	\draw[line width=1pt, black] (0.9,0.6) -- (1.35,2) node [midway, left] {\tiny B};
	\end{tikzpicture}
	}
    }  
\caption{Simplex and Pareto surface of $\mathcal{P}$ for face {\bf (f4)}.\\ Item $j$ can be assigned to any of the three players; the second item $k$ can be shared between two players. Six different allocations are generated; they produce a \textbf{trapezoid}, with the two extra points lying on its larger base.}
\label{face_f4}
\end{figure}
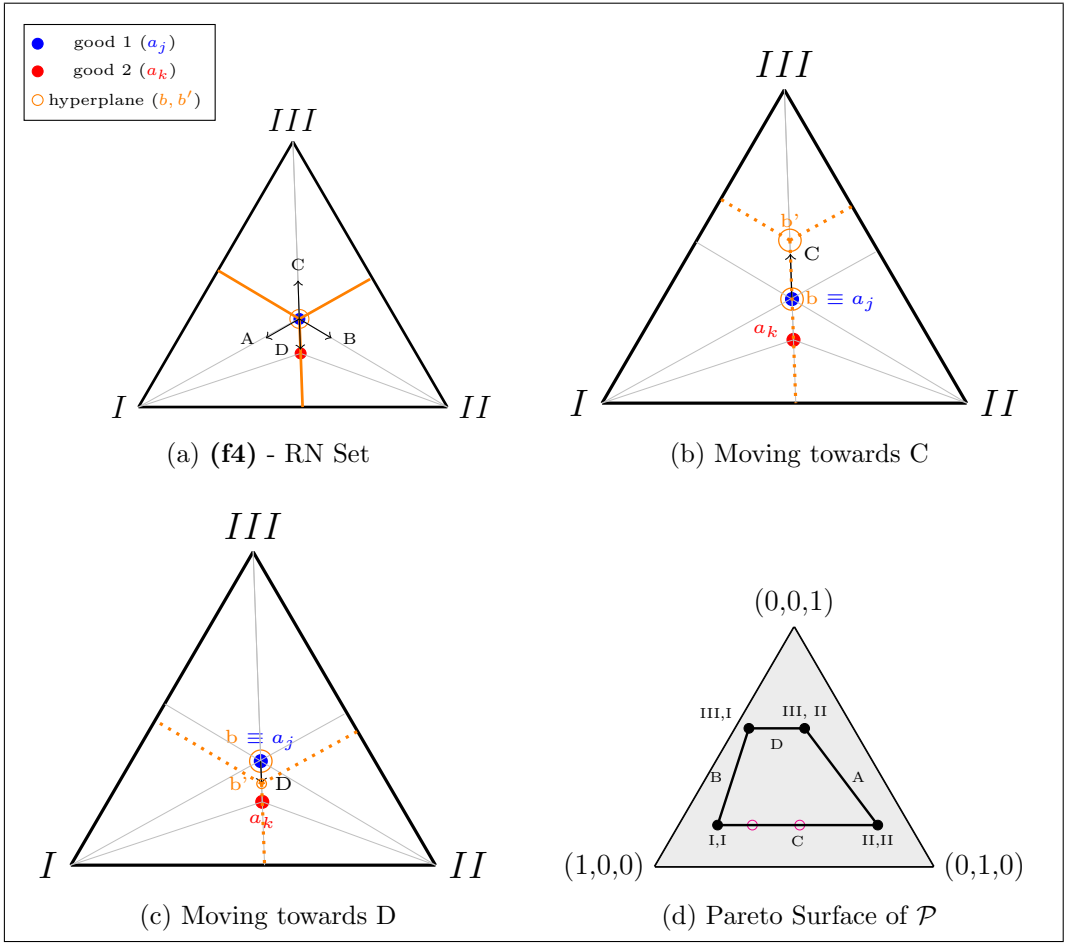

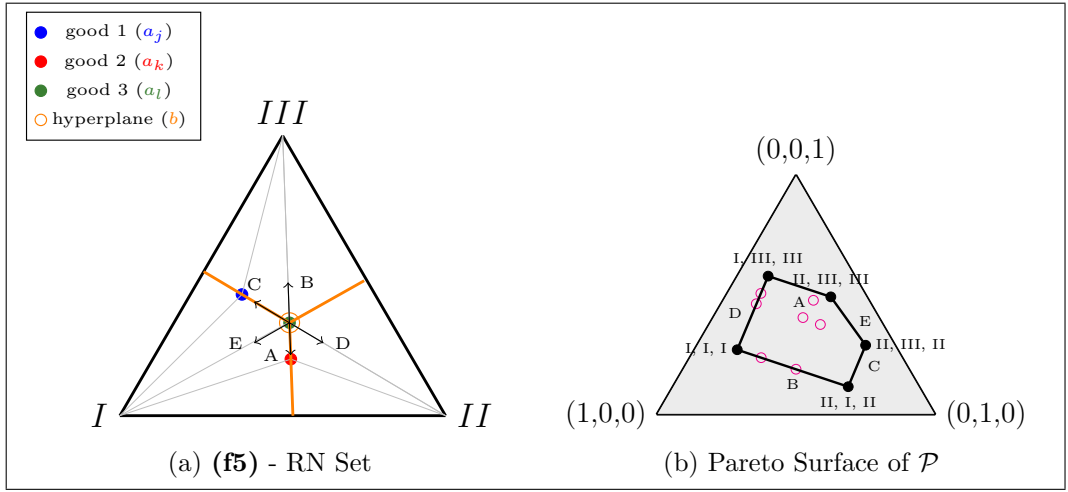
\begin{figure}[!htb]
    \subfloat[\textbf{(f5)} - RN Set\label{subfig-1:f5}]{%
        \resizebox{0.48\textwidth}{!}{
	\begin{tikzpicture}
	\centering
	\draw[line width=1pt] (0,0)--(4,0);
	\coordinate [label=left:$I$] (pl1) at (0,0);
	\draw[line width=1pt] (0,0)--(2,3.4641);
	\coordinate [label=above:$III$] (pl3) at (2,3.4641);
	\draw[line width=1pt] (2,3.4641)--(4,0);
	\coordinate [label=right:$II$] (pl2) at (4,0);
	\filldraw[OliveGreen] (2.084,1.15) circle (2pt);
	\draw[orange] (2.084,1.15) circle (3.5pt);
	\filldraw[blue] (1.5,1.5) circle (2pt);
	\filldraw[red] (2.1,0.7) circle (2pt);
	\draw[line width=0.1pt, lightgray] (0,0)--(1.5,1.5);
	\draw[line width=0.1pt, lightgray] (2,3.4641)--(1.5,1.5);
	\draw[line width=0.1pt, lightgray] (4,0)--(1.5,1.5);
	\draw[line width=0.1pt, lightgray] (0,0)--(2.1,0.7);
	\draw[line width=0.1pt, lightgray] (2,3.4641)--(2.1,0.7);
	\draw[line width=0.1pt, lightgray] (4,0)--(2.1,0.7);
	\draw[line width=0.1pt, lightgray] (0,0)--(3,1.67);
	\draw[line width=0.1pt, lightgray] (2,3.4641)--(2.1253,0);
	\draw[line width=0.1pt, lightgray] (4,0)--(1.03,1.7825);
	\draw[line width=1pt, orange] (2.084,1.15)--(3,1.67);
	\draw[line width=1pt, orange] (2.084,1.15)--(2.1253,0);
	\draw[line width=1pt, orange] (2.084,1.15)--(1.03,1.7825);
	\draw[->] (2.0824,1.15) -- (2.0982,0.75) node [below, left] {\tiny A};
	\draw[->] (2.0824,1.15) -- (2.0662,1.65) node [below, right] {\tiny B};
	\draw[->] (2.0824,1.15) -- (1.65,1.41) node [above] {\tiny C};
	\draw[->] (2.0824,1.15) -- (2.5,0.9) node [below, right] {\tiny D};
	\draw[->] (2.0824,1.15) -- (1.65,0.9) node [below, left] {\tiny E};
	\begin{customlegend}[
	legend entries={good 1 (\textcolor{blue}{$a_j$}), good 2 (\textcolor{red}{$a_k$}), good 3 (\textcolor{OliveGreen}{$a_l$}), hyperplane (\textcolor{orange}{$b$})},
	legend style={at={(1,5)}, font=\tiny}
	]
    	\addlegendimage{blue, only marks}
    	\addlegendimage{red, only marks}
    	\addlegendimage{OliveGreen, only marks}
	\addlegendimage{mark=o, only marks, color=orange}
    	\end{customlegend}
	\end{tikzpicture}
	}
    }
    \hfill
    \subfloat[Pareto Surface of $\mathcal{P}$\label{subfig-2:f5}]{%
        \resizebox{0.48\textwidth}{!}{
	\begin{tikzpicture}
	\centering
	\draw[line width=1pt] (0,0)--(4,0);
	\coordinate [label=left:\text{(1,0,0)}] (pl1) at (0,0);
	\draw[line width=1pt] (0,0)--(2,3.4641);
	\coordinate [label=above:\text{(0,0,1)}] (pl3) at (2,3.4641);
	\draw[line width=1pt] (2,3.4641)--(4,0);
	\coordinate [label=right:\text{(0,1,0)}] (pl2) at (4,0);
	\draw [black, fill=light-gray] (0,0) -- (4,0) -- (2,3.4641) -- (0,0);
	\filldraw[black] (1.6,2) circle (2pt);
	\coordinate [label=above:\tiny \text{I, III, III} ] (vertex1) at (1.6,2);
	\filldraw[black] (2.5,1.7) circle (2pt);
	\coordinate [label=above:\tiny \text{II, III, III} ] (vertex2) at (2.5,1.7);
	\filldraw[black] (3, 1) circle (2pt);
	\coordinate [label=right:\tiny \text{II, III, II} ] (vertex3) at (3, 1);
	\filldraw[black] (2.75, 0.4) circle (2pt);
	\coordinate [label=below:\tiny \text{II, I, II} ] (vertex4) at (2.75, 0.4);
	\filldraw[black] (1.156, 0.935) circle (2pt);
	\coordinate [label=left:\tiny \text{I, I, I} ] (vertex5) at (1.156, 0.935);
	\draw[Magenta] (1.5,0.82) circle (2pt);
	\draw[Magenta] (2,0.653) circle (2pt);
	\draw[Magenta] (1.496,1.75) circle (2pt);
	\draw[Magenta] (1.43,1.6) circle (2pt);
	\draw[Magenta] (2.10,1.4) circle (2pt);
	\draw[Magenta] (2.25,1.65) circle (2pt);
	\draw[Magenta] (2.35,1.3) circle (2pt);
	\draw[line width=1pt, black] (1.6,2) -- (2.5,1.7) node [midway, below] {\tiny A};
	\draw[line width=1pt, black] (2.5,1.7) -- (3, 1) node [midway, right] {\tiny E};
	\draw[line width=1pt, black] (3, 1) -- (2.75, 0.4) node [midway, right] {\tiny C};
	\draw[line width=1pt, black] (2.75, 0.4) -- (1.156, 0.935) node [midway, below] {\tiny B};
	\draw[line width=1pt, black] (1.156, 0.935) -- (1.6,2) node [midway, left] {\tiny D};
	\end{tikzpicture}
	}
    }
\caption{Simplex and Pareto surface of $\mathcal{P}$ for face {\bf (f5)}.\\ Item $j$ can be assigned to any of the three players; the second item $k$ can be shared between two players. The twelve different allocations produce a \textbf{pentagonal} face with five vertices and with two pairs of parallel edges with unequal length. The largest edge of each pair contains two additional points, while the remaining three points are in the interior of the face.}
\label{face_f5}
\end{figure}

\begin{figure}[!htb]
    \subfloat[\textbf{(f6)} - RN Set\label{subfig-1:f6}]{%
        \resizebox{0.48\textwidth}{!}{
	\begin{tikzpicture}
	\centering
	\draw[line width=1pt] (0,0)--(4,0);
	\coordinate [label=left:$I$] (pl1) at (0,0);
	\draw[line width=1pt] (0,0)--(2,3.4641);
	\coordinate [label=above:$III$] (pl3) at (2,3.4641);
	\draw[line width=1pt] (2,3.4641)--(4,0);
	\coordinate [label=right:$II$] (pl2) at (4,0);
	\filldraw[blue] (1.5,1.5) circle (2pt);
	\filldraw[red] (2.1,0.7) circle (2pt);
	\filldraw[OliveGreen] (2.5,1.375) circle (2pt);
	\filldraw[Plum] (2.084,1.15) circle (2pt);
	\draw[orange] (2.084,1.15) circle (3.5pt);
	\draw[line width=0.1pt, lightgray] (0,0)--(1.5,1.5);
	\draw[line width=0.1pt, lightgray] (2,3.4641)--(1.5,1.5);
	\draw[line width=0.1pt, lightgray] (4,0)--(1.5,1.5);
	\draw[line width=0.1pt, lightgray] (0,0)--(2.1,0.7);
	\draw[line width=0.1pt, lightgray] (2,3.4641)--(2.1,0.7);
	\draw[line width=0.1pt, lightgray] (4,0)--(2.1,0.7);
	\draw[line width=0.1pt, lightgray] (0,0)--(2.5,1.375);
	\draw[line width=0.1pt, lightgray] (2,3.4641)--(2.5,1.375);
	\draw[line width=0.1pt, lightgray] (4,0)--(2.5,1.375);
	\draw[line width=0.1pt, lightgray] (0,0)--(3,1.67);
	\draw[line width=0.1pt, lightgray] (2,3.4641)--(2.1253,0);
	\draw[line width=0.1pt, lightgray] (4,0)--(1.03,1.7825);
	\draw[line width=1pt, orange] (2.084,1.15)--(3,1.67);
	\draw[line width=1pt, orange] (2.084,1.15)--(2.1253,0);
	\draw[line width=1pt, orange] (2.084,1.15)--(1.03,1.7825);
	\draw[->] (2.0824,1.15) -- (2.0982,0.75) node [midway] {\tiny A};
	\draw[->] (2.0824,1.15) -- (2.0662,1.65) node [midway] {\tiny B};
	\draw[->] (2.0824,1.15) -- (1.65,1.41) node [midway] {\tiny C};
	\draw[->] (2.0824,1.15) -- (2.5,0.9) node [midway] {\tiny D};
	\draw[->] (2.0824,1.15) -- (1.65,0.9) node [midway] {\tiny E};
	\draw[->] (2.0824,1.15) -- (2.45,1.3475) node [midway] {\tiny F};
	\begin{customlegend}[
	legend entries={good 1 (\textcolor{blue}{$a_j$}), good 2 (\textcolor{red}{$a_k$}), good 3 (\textcolor{OliveGreen}{$a_l$}), good 4 (\textcolor{Plum}{$a_m$}), hyperplane (\textcolor{orange}{$b$})},
	legend style={at={(1,5)}, font=\tiny}
	]
    	\addlegendimage{blue, only marks}
    	\addlegendimage{red, only marks}
    	\addlegendimage{OliveGreen, only marks}
    	\addlegendimage{Plum, only marks}
	\addlegendimage{mark=o, only marks, color=orange}
    	\end{customlegend}
	\end{tikzpicture}
	}
    }
    \hfill
    \subfloat[Pareto Surface of $\mathcal{P}$\label{subfig-2:f6}]{%
        \resizebox{0.48\textwidth}{!}{
	\begin{tikzpicture}
	\centering
	\draw[line width=1pt] (0,0)--(4,0);
	\coordinate [label=left:\text{(1,0,0)}] (pl1) at (0,0);
	\draw[line width=1pt] (0,0)--(2,3.4641);
	\coordinate [label=above:\text{(0,0,1)}] (pl3) at (2,3.4641);
	\draw[line width=1pt] (2,3.4641)--(4,0);
	\coordinate [label=right:\text{(0,1,0)}] (pl2) at (4,0);
	\draw [black, fill=light-gray] (0,0) -- (4,0) -- (2,3.4641) -- (0,0);
	\filldraw[black] (0.7,1.2126) circle (2pt);
	\coordinate [label=left:\tiny \text{III, III, I, III} ] (vertex1) at (0.7,1.2126);
	\filldraw[black] (3.3,1.2126) circle (2pt);
	\coordinate [label=right:\tiny \text{III, III, II, III} ] (vertex2) at (3.3,1.2126);
	\filldraw[black] (3.75,0.4332) circle (2pt);
	\coordinate [label=right:\tiny \text{III, II, II, II} ] (vertex3) at (3.75,0.4332);
	\filldraw[black] (3.5,0) circle (2pt);
	\coordinate [label=below:\tiny \text{I, II, II, II} ] (vertex4) at (3.5,0);
	\filldraw[black] (1,0) circle (2pt);
	\coordinate [label=below:\tiny \text{I, II, I, I} ] (vertex5) at (1,0);
	\filldraw[black] (0.5,0.866) circle (2pt);
	\coordinate [label=left:\tiny \text{I, III, I, I} ] (vertex6) at (0.5,0.866);
	\draw[Magenta] (1.2,0) circle (2pt);
	\draw[Magenta] (3.2,0) circle (2pt);
	\draw[Magenta] (3.62,0.65) circle (2pt);
	\draw[Magenta] (3.68,0.55) circle (2pt);
	\draw[Magenta] (0.635,1.1) circle (2pt);
	\draw[Magenta] (0.55,0.95) circle (2pt);
	\draw[line width=1pt, black] (0.7,1.2126) -- (3.3,1.2126) node [midway, above] {\tiny A};
	\draw[line width=1pt, black] (3.3,1.2126) -- (3.75,0.4332) node [midway, right] {\tiny E};
	\draw[line width=1pt, black] (3.75,0.4332) -- (3.5,0) node [midway, left] {\tiny C};
	\draw[line width=1pt, black] (3.5,0) -- (1,0) node [midway, above] {\tiny B};
	\draw[line width=1pt, black] (1,0) -- (0.5,0.866) node [midway, left] {\tiny F};
	\draw[line width=1pt, black] (0.5,0.866) -- (0.7,1.2126) node [midway, right] {\tiny D};
	\draw[->] (3.5,3) -- (2,1);
	\coordinate [label=above:\text{\tiny{+ 12 internal points}}] (ip) at (3.5,3);
	\end{tikzpicture}
	}
    }
\caption{Simplex and Pareto surface of $\mathcal{P}$ for face {\bf (f6)}.\\ Item $j$ can be assigned to any of the three players; the second item $k$ can be shared between two players. The twenty-four different distributions produce \textbf{hexagons} with parallel opposite sides of unequal length. The largest edge of each pair contains two additional points, while the remaining twelve points are in the interior of the face.}
\label{face_f6}
\end{figure}

The maximum number of faces is obtained when all items are s-independent, and all the intersections of the disputing segments do not coincide with goods in $RNS$. In such case only faces {\bf (f1)} and {\bf (f2)} are present, producing $k + \binom{k}{2}  = (k(k+1))/2$ faces.

The classification of the faces on PF enables us to bound the number of split items of any PO allocation.

\begin{theo} Under MAC,
every Pareto optimal allocation can be obtained under the following alternative conditions:
\begin{enumerate}[a)]
\item No good is split among the players
\item One good is split among at most three players
\item Two goods are split, each one between two players
\end{enumerate}
\end{theo}
\begin{proof}
Case $a$ occurs when the Pareto optimal allocation is also an integer one. Otherwise, the allocation belongs to (at least) one of the faces (f1) through (f6).
If the face is of type (f1) or (f2), there is nothing to prove, and we dealing with cases $b$ and $c$, respectively.

Consider now any allocation $\mathbf{X}$ with value $a(\mathbf{X})=(z_1,z_2,z_3)$ on a face (f3). 
Assume w.l.o.g.\ that the items are distributed according to the example in Figure \ref{face_f3}. 
For the sake of simplicity set $\gamma_1 = x_{11}$, $\gamma_2 = x_{22}$ and $\gamma_3= x_{13}$. 
Also, denote as $r_1,r_2$ and $r_3$ the players' allocation value for the remaining goods (other then the first three) obtained with a Pareto rule compatible with $\beta=(\beta_1,\beta_2,\beta_3)$. 
Any solution of the following linear system denotes an allocation compatible with $(z_1,z_2,z_3)$
\[
\left\{
\begin{array}{lcr}
\gamma_1  a_{11} + \gamma_3  a_{13} + r_1 & = & z_1
\\
\gamma_2 a_{22} + (1 - \gamma_3) a_{23} + r_2 & = & z_2
\\
(1 - \gamma _1)a_{31} + (1 - \gamma_2) a_{32} + r_3 & = & z_3 
\end{array}
\right.
\]
in the constraint region
\[
R_1 = \{ (\gamma_1,\gamma_2,\gamma_3) \in \R^3 : 0 \leq \gamma_1,\gamma_2,\gamma_3 \leq 1 \} \; .
\]
The system, after proper rearrangement, becomes 
\[
\left\{
\begin{array}{ccccl}
a_{11} \gamma_1  & & + a_{13} \gamma_3    & = & z_1 - r_1
\\
& a_{22} \gamma_2   & - a_{23} \gamma_3  & = & z_2
 - r_2 - a_{23}\\
 - a_{31} \gamma _1  &  - a_{32} \gamma_2 &  & = & z_3 - r_3 - a_{31} - a_{32} 
\end{array}
\right. \; .
\]
Consider now the coefficient matrix
\[
\label{Amat}
\mathbf{A}= \begin{bmatrix}
a_{11} & 0 & a_{13} 
\\
0 & a_{22} &  - a_{23}
\\
-a_{31} & -a_{32} & 0
\end{bmatrix}
\]
Since good 1 is disputed between Players $I$ and $III$, we have $\phi_{1} a_{11} = \phi_3 a_{31} > \phi_2 a_{21}\geq 0$, where $(\phi_1,\phi_2,\phi_3)=RD(\beta)$. 
Moving to goods 2 and 3, we have $\phi_{2} a_{22} = \phi_{3} a _{32} > \phi_1 a_{12} \geq 0$ and $\phi_1 a_{13} = \phi_2 a_{23} > \phi_3 a_{33} \geq 0$, respectively. 
Therefore, $\det(\mathbf{A})= a_{13} a_{22} a_{31} - a_{11} a_{23} a_{32}=0$,
$\det \left[\begin{smallmatrix}
a_{11} & 0
\\
0 & a_{22}
\end{smallmatrix} \right]= a_{11} a_{22} > 0$
 and $\mathrm{rank}(\mathbf{A})=2$.

Since the allocation lies in the face, the system is admissible, and the set of solutions forms a line in $\R^3$, which intersects the border of $R_1$ at least once (actually twice). 
Any face of the border describes situations where at most
two goods are split, each between two players, that is case $c$. By virtue of the last statement in $(ii)$ of Lemma \ref{edgesLemma} the proof can easily be extended to the case where several goods lie on the same disputing line.

Consider now an allocation with values $a(\mathbf{X})=(z_1,z_2,z_3)$ on an (f4) type of face. Considering w.l.o.g.\ the distribution of goods described in Figure \ref{face_f4}, and denoting for simplicity $\xi_1=x_{11}$, $\xi_2 = x_{21}$ and $\gamma=x_{12}$, any allocation of goods 1 and 2 compatible with $(z_1,z_2,z_3)$ is a solution of the following system
\[
\left\{
\begin{array}{lcr}
\xi_1 a_{11}+ \gamma a_{12} + r_1 & = & z_1
\\
\xi_2 a_{21} + (1 - \gamma) a_{22} + r_2 & = & z_2
\\
 (1 - \xi_1 - \xi_2) a_{31} + r_3 & = & z_3 
\end{array}
\right.
\]
in the constraint region (see Figure \ref{proof_f4})
\[
R_2 = \{ (\xi_1,\xi_2,\gamma) \in \R^3 : 0 \leq \gamma \leq 1, \xi_1,\xi_2 \geq 0, \xi_1 + \xi_2 \leq 1 \}
\]
After proper rearrangement, the linear system becomes
\[
\left\{
\begin{array}{ccccl}
 a_{11} \xi_1 &   & + a_{12} \gamma & = & z_1 - r_1
\\
 & a_{21} \xi_2& -a_{22} \gamma & = & z_2
 - r_2 - a_{22}\\
-a_{31} \xi_1& -a_{31} \xi_2 &    & = & z_3 - r_3 - a_{31} 
\end{array}
\right.
\]
Similarly to the previous case the coefficient matrix has rank 2 and the admissible system has a set of solutions given by a line in $\R^3$, which intersects the border of the closed and bounded admissibility region $R_2$ at least once. $R_2$ has 5 faces. In case the solution belongs to one of the two faces (colored in gray in Figure \ref{proof_f4})
\[
\gamma = 0 \mbox{ or } 1, \quad \xi_1,\xi_2 \geq 0, \xi_1 + \xi_2 \leq 1
\]
we fall on case $b$. For the other 3 faces, namely
\begin{gather*}
\xi_1 = 0, \quad 0 \leq \gamma,\xi_2 \leq 1
\\
\xi_2 = 0, \quad 0 \leq \gamma,\xi_1 \leq 1
\\
\xi_1 + \xi_2 = 1, \quad 0 \leq \gamma \leq 1
\end{gather*}
we fall on case $c$. Notice that in the third case (corresponding to the shaded rectangle in Figure \ref{proof_f4}), item 1 is split between Players 1 and 2.
As before, the extension to several goods on the same disputing line is guaranteed by Lemma \ref{edgesLemma}.

Moving to face (f5), we consider the goods' distribution described in Figure \ref{face_f5}. Setting $\gamma_1 = x_{11}$, $\gamma_2 = x_{12}$, $\xi_1 = x_{13}$ and $\xi_2 = x_{23}$ we must find a solution of
\[
\left\{
\begin{array}{lcr}
\gamma_1 a_{11} + \gamma_2 a_{12} + \xi_1 a_{13} + r_1 & = & z_1
\\
(1 - \gamma _1) a_{21} + \xi_2 a_{23} + r_2 & = & z_2
\\
(1 - \gamma_2) a _{32} + (1 - \xi_1 - \xi_2) a_{33} + r_3 & = & z_3 
\end{array}
\right.
\]
in the constraint region
\[
R_3 = \{ (\xi_1,\xi_2,\gamma_1,\gamma_2) \in \R^4 : 0 \leq \gamma_1,\gamma_2 \leq 1, \xi_1,\xi_2 \geq 0, \xi_1 + \xi_2 \leq 1 \}
\]
Rearranging the terms of the system, we have
\[
\left\{
\begin{array}{cccccl}
 a_{11} \gamma_1& + a_{12} \gamma_2& + a_{13} \xi_1   &  & = & z_1 - r_1
\\
-a_{21} \gamma_1 & & &  +a_{23} \xi_2 & = & z_2 - r_2 -a_{21}\\
& -a_{32} \gamma_2& -a_{33} \xi_1  & -a_{33} \xi_2   & = & z_3 - r_3 -a_{32} - a_{33}  
\end{array}
\right.
\]
The system is admissible with rank 2, and the set of solution forms a two-dimensional subspace in $\R^4$. This set will intersect\footnote{Each bidimensional object is the intersection of two hyperplanes. 
The intersection between two bidimensional objects will be the intersection of four hyperplanes in $\R^4$ -- typically a point} at least one of the two dimensional faces composing the border of the closed and bounded constraint region $R_3$. If this solution belongs to one of the following faces
\[
\left.
\begin{array}{c}
\gamma_1 =0\mbox{ or } 1
\\
\gamma_2 =0\mbox{ or } 1
\end{array}
\right\}
\qquad \xi_1,\xi_2 \geq 0, \xi_1 + \xi_2 \leq 1
\]
we fall on case $b$. For the remaining faces, namely
\begin{gather*}
\left.
\begin{array}{c}
\gamma_i =0\mbox{ or } 1 \quad (i=1,2)
\\
\xi_h =0\mbox{ or } 1 \quad (h=1,2)
\end{array}
\right\}
\qquad
0 \leq \gamma_j, \xi_k \leq 1 \quad  (j \neq i, k \neq h)
\\
\xi_1=\xi_2=0, \qquad 0 \leq \gamma_1,\gamma_2 \leq 1
\\
\left.
\begin{array}{c}
\xi_1 + \xi_2 =1 
\\
\gamma_i =0\mbox{ or } 1 \quad (i=1,2)
\end{array}
\right\}
\qquad
0 \leq \gamma_j \leq 1 \quad (j \neq i)
\end{gather*}
We fall on case $c$. In the last situation, item 3 is split between Players 1 and 2 and item $i$ is split among the corresponding Players. The usual extension to a larger number of item applies.

Suppose now that the allocation value $(z_1,z_2,z_3)$ belongs to face (f6). Following the distribution in Figure \ref{face_f6}, we let $\gamma_1 = x_{11}$, $\gamma_2 = x_{22}$, $\gamma_3 = x_{13}$, $\xi_1 = x_{14}$ and $\xi_2 = x_{24}$, and we solve
\[
\left\{
\begin{array}{lcr}
 \gamma_1 a_{11} + \gamma_3 a_{13} + \xi_1 a_{14} + r_1 & = & z_1
\\
\gamma_2 a_{22} +(1 - \gamma_3) a_{23} + \xi_2 a_{24} + r_2 & = & z_2
\\
(1 - \gamma_1) a_{31} + (1 - \gamma_2) a_{32} + (1 - \xi_1 - \xi_2) a_{34} + r_3 & = & z_3 
\end{array}
\right.
\]
in the constraint region
\[
R_4 = \{ (\xi_1,\xi_2,\gamma_1,\gamma_2,\gamma_3) \in \R^5 : 0 \leq \gamma_1,\gamma_2,\gamma_3 \leq 1, \xi_1,\xi_2 \geq 0, \xi_1 + \xi_2 \leq 1 \}
\]
Rearranging the terms, the system becomes
\[
\left\{
\begin{array}{ccccccl}
a_{11} \gamma_1& & a_{13} \gamma_3 & + a_{14} \xi_1  &  & = & z_1 - r_1
\\
&  a_{22} \gamma_2 & - a_{23} \gamma_3& & + a_{24} \xi_2 & = & z_2 - r_2 - a_{23}\\
- a_{31} \gamma_1& -a_{32} \gamma_2& & -a_{34} \xi_1 & -a_{34} \xi_2  & = & z_3 - r_3 -a_{31} - a_{32} - a_{34}  
\end{array}
\right.
\]
The system is admissible with rank 2, and the solution set is now tridimensional. It must then intersect\footnote{Here we consider the insersection of one bidimensional object with a tridimensional one} at least one of the following bidimensional faces of the closed and bounded region
\begin{gather*}
\left.
\begin{array}{c}
\gamma_1 =0\mbox{ or } 1 
\\
\gamma_2 =0\mbox{ or } 1
\\
\gamma_3 =0\mbox{ or } 1
\end{array}
\right\}
\qquad
\xi_1,\xi_2 \geq 0, \xi_1 + \xi_2 \leq 1
\\
\left.
\begin{array}{c}
\xi_i =0\mbox{ or } 1 \quad (i=1,2) 
\\
\gamma_j =0\mbox{ or } 1 \quad (j=1,2,3)
\\
\gamma_h =0\mbox{ or } 1 \quad (h \neq j)
\end{array}
\right\}
\qquad
0 \leq \xi_k,\gamma_{\ell} \leq 1 \quad (k \neq i, \ell \neq j,h)
\\
\left.
\begin{array}{c}
\xi_1 + \xi_2 = 1
\\
\gamma_j =0\mbox{ or } 1 \quad (j=1,2,3)
\\
\gamma_h =0\mbox{ or } 1 \quad (h \neq j)
\end{array}
\right\}
\qquad
0 \leq \gamma_{\ell} \leq 1 \quad (\ell \neq j,h)
\end{gather*}
The first type of face describes a type $b$ situation, while all the other describe a type $c$ situation. The usual extension to a larger number of item applies.

\end{proof}
\begin{figure}[htp]
  \centering
    \includegraphics[width=0.3\textwidth]{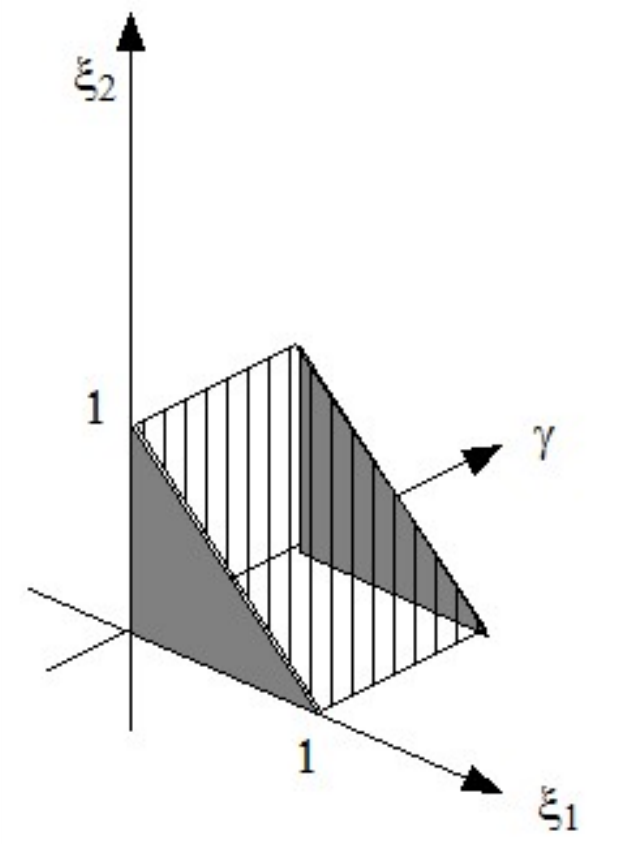}
  \caption{The constraint region for face {\bf (f4)}}
\label{proof_f4}
\end{figure}

\section{The Pareto Boundary as a graph}

We now show that $RNS$ can be used to build a graph $\mathcal{G}= \{V,E\}$ where each vertex $v \in V$ is a face on the Pareto surface, and two vertices $v_i$ and $v_j$ are connected by an arc\footnote{We prefer to use ``arc'' in place of the more common ``edge'' to avoid confusion with the edges of a face on the Pareto surface.} $e_{ij} \in E$ if and only if the corresponding faces are adjacent, i.e.\ they share a common edge.
The idea to build this graph is simply to consider all the goods in $RNS$ with their disputing segments. The vertices $V$ will consist of all the points in $RNS$ coinciding with a good or with an intersection of the dividing lines. Two vertices $v_i,v_j \in V$ will be connected by an arc $e_{ij} \in E$ whenever there is supporting segment joining the two vertices, with no other vertex of $V$ in between. We refer to Figure \ref{proof_adjTh}(a) for an example of such graph.
Theorem \ref{facesTh} shows that each vertex in $V$ represents a face on the Pareto surface. We now show that two vertices are adjacent (i.e., share an arc) whenever the corresponding faces on the Pareto surface are adjacent (i.e., share an edge).

\begin{theo}
\label{adjTheorem} Under MAC,
two faces on the upper border of $\mathcal{P}$ are adjacent, i.e.\ they share a common line segment, if and only if the corresponding vertices $v_k$ and $v_{\ell}$ are joined by an arc in $\mathcal{G}$.
\end{theo}

\begin{proof}
First of all we  prove that adjacent faces on the Pareto surface correspond to vertices lying on the same disputing segment. In fact, each edge of the faces {\bf (f1)} through {\bf (f6)} 
is obtained when are assigned first to one player and then to another, while the allocation of the remaining goods remains unchanged. In order for two different hyperplanes to support the same edge, both hyperplanes must have the same disputing segment in common. Therefore one of them must lie on the supporting segment of the other, and, as a consequence they have to be  connected by a disputing segment.
Also, two faces cannot be associated with vertices on the same supporting segment separated by a third vertex. For, in such a case they would have no allocation and, a fortiori, no edge in common.

We now show that every two adjacent vertices on $\mathcal{G}$ correspond to adjacent faces on the Pareto surface. In fact, consider any two adjacent vertices $v_i.v_j \in V$ and take the hyperplane placed on the midpoint of the arc $e_{ij}$ connecting the two vertices. Whatever the type of face represented by $v_i$ and $v_j$, the hyperplane touches the Pareto surface on an edge obtained by allocating the goods on both sides of the disputing segment aligned with the edge. The two faces associated to $v_i$ and $v_j$ (respectively) have the same edge in common, since they are both compatible with the allocation of the goods provided by the midpoint hyperplane. We refer to Figure \ref{proof_adjTh}(b) for an illustration of the proof.

\begin{figure}[htp]
\centering
\includegraphics[scale=0.45]{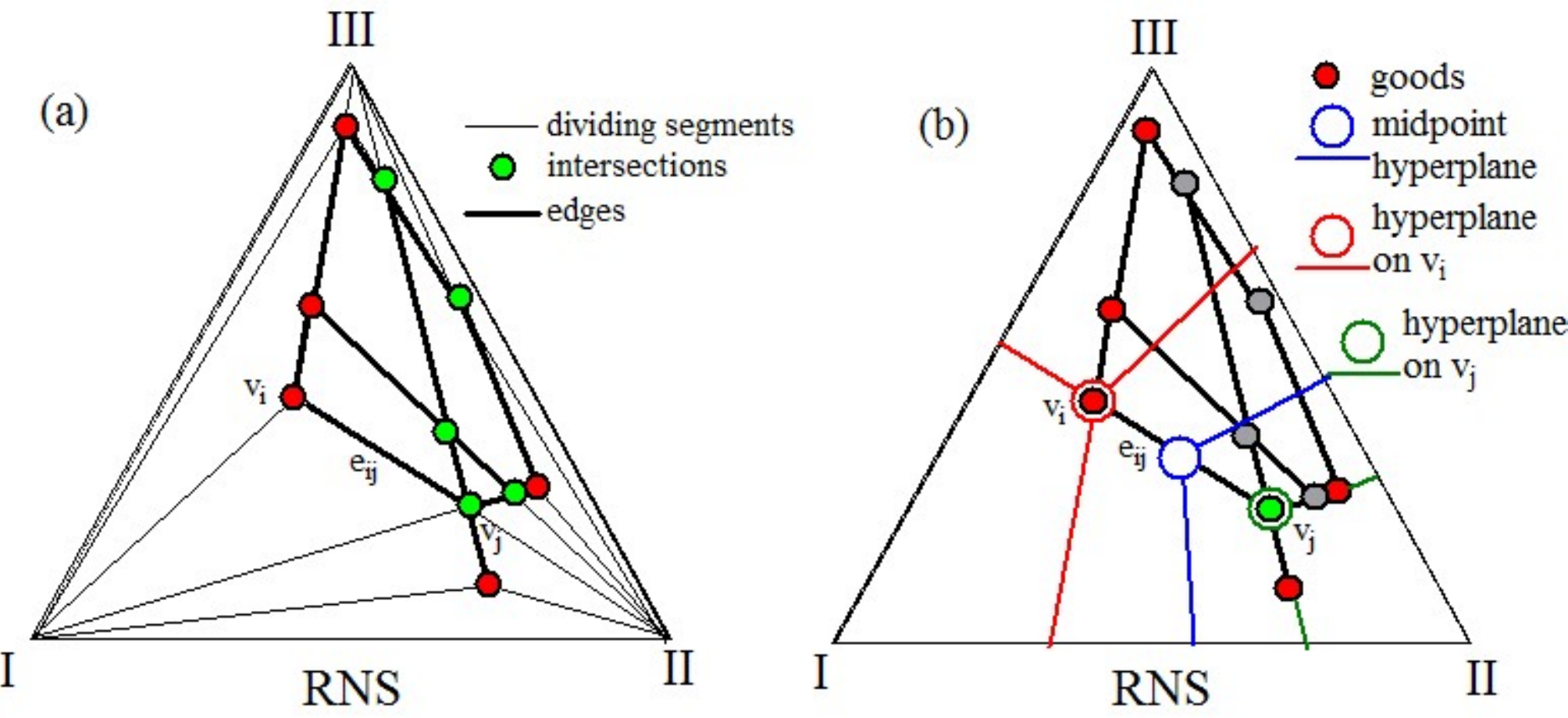}
\caption{(a) The graph formed by 5 goods. (b) A sketch of the last part of the proof in Theorem \ref{adjTheorem}. Three hyperplanes, those centered at $v_i$, $v_j$ and their midpoint, share the same two allocations, with all the goods but that in $v_i$ allocated according to the Pareto rule of the midpoint hyperplane, and the good in $v_i$ allocated first to player 1 and then to player 3 }\label{proof_adjTh}
\end{figure}

\end{proof}

\section{Algorithms}
\subsection{A simple algorithm}
We now consider a function $\tilde{g}:RNS \to [0,1]$ defined for each $\beta \in RNS$ as
\[
\tilde{g}(\beta)=g(RD(\beta))
\]
this function, in particular is defined for every vertex of the graph $\tilde{g}(v)$. The following Theorem shows that it suffices to check the value of $\tilde{g}$ on the face/vertices only. Moreover, the function $\tilde{g}$ inherits the convexity property of $g$ according to which it suffices to show that $\tilde{g}$ is a "local" minimum (i.e. a minimum w.r.t.\ the adjacent vertices) to make it a global optimum.
\begin{theo}
\label{loc_glob_gtilde}
Under MAC, a face/vertex $v^*$ containing the egalitarian ray has the following properties
\begin{enumerate}[i)]
\item It is the global minimum for $g$
\begin{equation}
\label{tildeg_absmin}
\tilde{g}(v^*) \leq \tilde{g}(v) \qquad \mbox{for any } v \in V 
\end{equation}
\item It suffices to show that it is a local minimum for $g$:
\begin{equation}
\label{tildeg_relmin}
\tilde{g}(v^*) \leq \tilde{g}(v') \qquad \mbox{for any } v' \mbox{ adjacent to } v^* 
\end{equation}
\end{enumerate}
\end{theo}
\begin{proof}
$(i)$ The PO-EQ allocation belongs to one or more Pareto faces. According to Theorem \ref{th_dallaglio} $(iii)$ one of the miniming argoments of $g$ (and therefore of $\tilde{g}$) will correspond to such Pareto face. The corresponding vertex on the graph will be associated to the same absolute minimum.

$(ii)$ We need to prove two preliminary claims.

We consider two hyperplanes: $\mathcal{H}(\mathrm{face})$ passing through a Pareto face, and $\mathcal{H}(\mathrm{edge})$ passing through an edge of the face. Denote with $g(\mathrm{face})$, $g(\mathrm{edge})$ resp., the value of $g$ corresponding to $\mathcal{H}(\mathrm{face})$, $\mathcal{H}(\mathrm{edge})$ resp., and denote with $\ell_{ed}$ the line  intersection of the the two (non parallel) hyperplanes containing the edge. Consider now the following projections on NPB, obtained by normalizing the points in the geometrical object: $p_{face}$, $p_{ed}$, $p_{eq}$, projections of the Pareto face, the line $\ell_{ed}$ and the bisector, respectively.
We prove the following claims:

\begin{description}
\item[Claim 1a] If $p_{face}$ and $p_{eq}$ are on the same side $p_{ed}$  (see Figure \ref{fig_claim1}, left) then $g(\mathrm{edge}) \geq g(\mathrm{face})$;
\item[Claim 1b] If $p_{face}$ and $p_{eq}$ are on opposite sides of $p_{ed}$ (see Figure \ref{fig_claim1}, right) then $g(\mathrm{edge}) \leq g(\mathrm{face})$;
\end{description}
\begin{figure}[htp]
  \centering
    \includegraphics[width=0.8\textwidth]{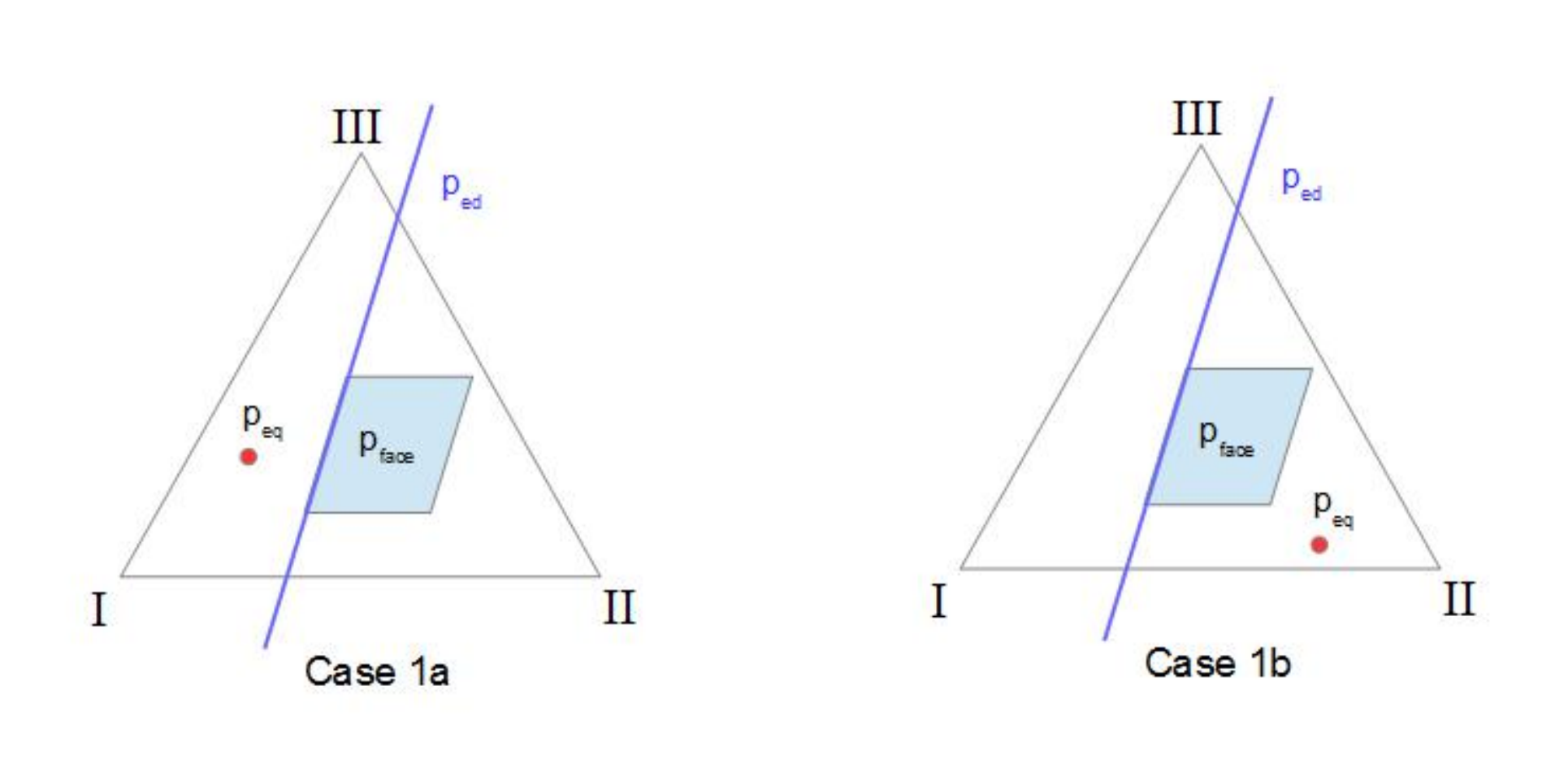}
  \caption{The two cases for claim 1}
\label{fig_claim1}
\end{figure} 
\begin{proof}[Proof of Claims 1a and 1b]
$(1a)$ Suppose $g(\mathrm{edge}) < g(\mathrm{face})$. Then the hyperplane $\mathcal{H}(\mathrm{edge})$ passes through the edge and $g(\mathrm{edge})(1,1,1)$, separating $\mathcal{P}$ from the Pareto face and $g(\mathrm{face})(1,1,1)$ (see Figure \ref{fig_proofclaim1}, left)  -- a contradiction.
\begin{figure}[htp]
  \centering
    \includegraphics[width=0.8\textwidth]{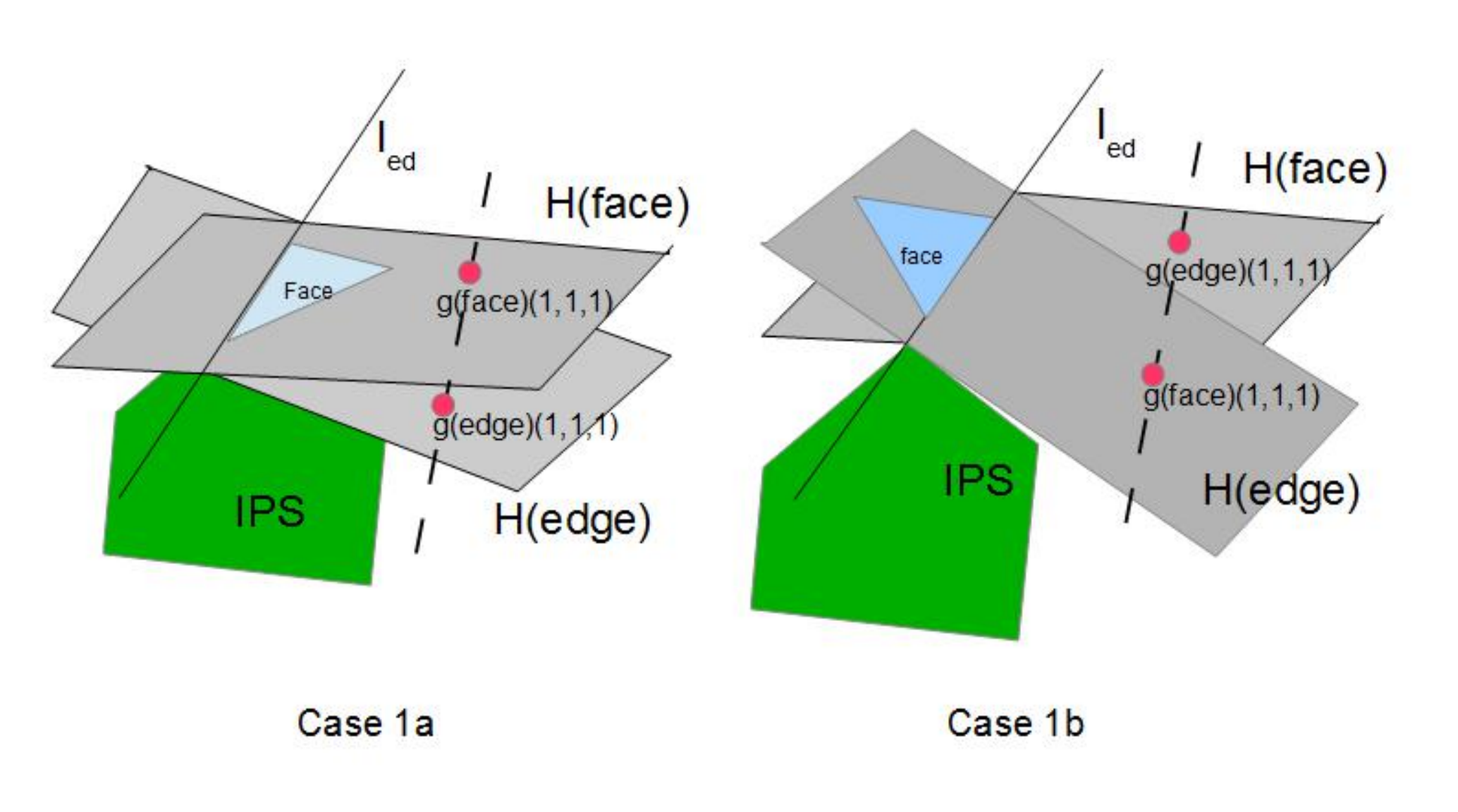}
  \caption{Proofs of Claims 1a and 1b}
\label{fig_proofclaim1}
\end{figure} 
$(1b)$ Suppose $g(\mathrm{edge}) > g(\mathrm{face})$. Then the hyperplane $\mathcal{H}(\mathrm{edge})$ passes through the edge and $g(\mathrm{edge})(1,1,1)$, separating $\mathcal{P}$ from the Pareto face and $g(\mathrm{face})(1,1,1)$ (see Figure \ref{fig_proofclaim1}, right) -- again a contradiction.
\end{proof}
As a consequence of these claims suppose    $\tilde{g}(v') < \tilde{g}(v'') $ for two adjacent faces, then, projecting the Pareto face $v'$, the line containing the edge and the bisector on NPB, the projected Pareto face and the projected bisector must lie on the same  side of the projected line.

{\bf Claim 2.} Suppose $\tilde{g}(v') = \tilde{g}(v'') $ for two adjacent faces. Then the bisector intesects the line generated by the common edge.

\begin{proof}[Proof of Claim 2]
Denote with $\mathcal{H}(v')$, $\mathcal{H}(v'')$ resp., the hyperplane passing through face $v'$, face $v''$ resp. Since $v' \neq v''$, the two hyperplanes are neither parallel nor coincident and intesect in a line $\ell_{ed}$ that includes the common edge. Suppose the bisector does not meet the line. Then the two hyperplanes $\mathcal{H}(v')$ and $\mathcal{H}(v'')$ have more than three non aligned points in common: those in $\ell_{ed}$ and $g(v')(1,1,1)$. Being distinct hyperplanes this is impossible. The bisector must meet the common line  $\ell_{ed}$.
\end{proof}
To prove the theorem we distinguish three cases:
\begin{description}
\item[Case 1] \eqref{tildeg_relmin} holds with strict inequality sign for all the adjacent edges.

Consider the projection of the the faces and the bisector on NPB. For each adjacent edge, the Pareto face and the bisector lie on the same side of the line generated by the edge. The projected bisector must belong to  the projected bisector and the same is true on the Pareto Boundary. By  theorem \ref{th_dallaglio} $(iii)$ $v^*$ is an absolute minimum of $g$, and therefore of $\tilde{g}$.

\item[Case 2] \eqref{tildeg_relmin} holds with strict inequality sign for all the adjacent edges, but one $\tilde{g}(v^*)=\tilde{g}(v'')$.

Considering the projections on NPB, the bisector and the Pareto face $v^*$ must lie on the same side of each line generated by all the adjacent edges different from $v''$. Moreover the bisector must lie on the (projected) line generated by the edge between the faces $v^*$ and $v''$. Once again the bisector meets the Pareto face $v^*$, and the theorem holds.

\item[Case 3] \eqref{tildeg_relmin} holds with  $\tilde{g}(v^*)=\tilde{g}(v'')$ for two or more adjacent vertices $v''$.

The bisector must belong to all the lines generates by the edges of the adjacent faces for which equality holds. The bisector belongs to their intersection which must be a vertex of the face. Once again the bisector meets the Pareto face $v^*$, and the theorem holds.

\end{description}

\end{proof}
The theorem  suggests the following simple algorithm

\subsection{A simple algorithm}
The following algorithm is based on theorem \ref{loc_glob_gtilde}.
\begin{description}
\item[Beginning ]Start from any $v^0 \in V$ (For instance the one closest to the center in RNS)

\item[Body] For the current $v^k \in V$, compute the value $\tilde{g}(v')$ any adjacent vertex $v'$:
\begin{itemize}
\item If \eqref{tildeg_relmin} holds $\Longrightarrow$ $v^k$ is optimal $\Longrightarrow$ End
\item Otherwise move towards the adjacent vertex $v^{k+1}$ with lowest value of $\tilde{g}$. $\Longrightarrow$ Repeat step with $v^{k+1}$.
\end{itemize}
\item[End] From the optimal vertex/face $\Longrightarrow$ find the optimal allocation.
\end{description}
\begin{theo}
The simple algorithm converges in a finite number of steps.
\end{theo}
\begin{proof}
There is a finite number of vertices in the graph, and the algorithm cannot cycle.
\end{proof}

\subsection{Steepest descent along the graph}
Moving along the edge from $(c_1,c_2,c_3)$ to $\left( c_1+\varepsilon,c_2 - \frac{c_2}{c_2 + c_3} \varepsilon, c_3 - \frac{c_3}{c_2 + c_3} \varepsilon \right)$, with $0 < \varepsilon < c_2 + c_3$, note that 
\[
\frac{c_2 - \frac{c_2}{c_2 + c_3} \varepsilon}{c_3 - \frac{c_3}{c_2 + c_3} \varepsilon} =
\frac{c_2 \left( 1 - \frac{\varepsilon}{c_2 + c_3}\right)}{c_3 \left( 1 - \frac{\varepsilon}{c_2 + c_3} \right)}=\frac{c_2}{c_3}
\]
i.e.\ the ratio remains the same, and the new point is on the supporting segment joining $(c_1,c_2,c_3)$ to the vertex of Player $I$. 

We now investigate what happens to the corresponding hyperplane coefficients.

\begin{lem}
If $S(\varepsilon)=(s_1,s_2,s_3)$ denotes the variation of the hyperplane coefficients, i.e.
\[
S(\varepsilon):= RD \left( c_1+\varepsilon,c_2 - \frac{c_2}{c_2 + c_3} \varepsilon, c_3 - \frac{c_3}{c_2 + c_3} \varepsilon \right) - RD(c_1,c_2,c_3)
\]
then
\[
S(\varepsilon) = \delta(\varepsilon) (-c_2-c_3,c_2,c_3)
\]
where
\[
\delta(\varepsilon)=\frac{c_2 c_3 \varepsilon}{(c_2 c_3 + c_1 c_2 + c_1 c_3)(c_1(c_2+c_3) +c_2^2 c_3 + c_2 c_3^2 + \varepsilon (c_3^2 + c_2^2 + c_2 c_3))}
\]
\end{lem}
\begin{proof}
Verify it with Mathematica
\end{proof}
The shifts along the RNS cause an analogous shift in the hyperplane coefficients, but in the opposite direction: To an increase in $c_1$, there corresponds a decrease in $b_1$, and, conversely, a decrease in $c_2$ and $c_3$ determines an increase in the corresponding hyperplane coefficients $b_2$ and $b_3$.

The shift in $b$ is not linear. This implies that the value of the function $g$ will not change linearly. Instead of shoosing an arbitary value for $\varepsilon$, we may consider the instantaneous rate of change at $\varepsilon =0$
\[
\delta'(0)=\frac{c_2 c_3}{(c_2 + c_3)(c_2 c_3 + c_1 (c_2 + c_3))^2}
\]
Therefore we will consider the following variations in the hyperplane coefficients
\[
\delta'(0) (-c_2 -c_3, c_2,c_3)
\]
Note also that a similar in the other direction is obtained by reversing signs. Similarly, shifts along other supporting signs are obtained by replacing the roles of $c_1$, with that of $c_2$ ($c_3$, resp.) if the supporting segment toward Player $II$ (Player $III$, resp.) is considered.

Whatever shift is considered, denote with $s_1$, $s_2$ and $s_3$ the variation in the hyperplane coefficients, with $s_1+s_2+s_3=0$. Assuming that the shift does not take to a point out of RNS, the variation in $g$ is given by
\[
\Delta \tilde{g} = \sum_{i \in N} s_i a_i(\mathbf{X}_i^-) + \sum_{j \in dg(c)} \max_{i \in dp(j)}\{s_i a_{ij}\}
\] 
where $\mathbf{X}_i^-$ denotes the goods that, according to $PAR(c)$ are allocated to Player $i$ without disputing (and do not belong to a disputing line of $c$), $dg(c)$ denotes the set of disputing goods in $c$, and $dp(j)$ denotes the set of disputing players for good $j$.

In case $(s_1,s_2,s_3)=\delta'(0)(-c_2-c_3,c_2,c_3)$ we can define the directional derivative $\frac{dg}{d\vec{\ell}}=\Delta \tilde{g}$ of $\tilde{g}$ in the direction $\vec{\ell}$ towards pl.I. Similar definitions can be given for the other admissible directions from a given face/vertex $v$.

\subsection{A more subtle algorithm}
The following algorithm is based on a steepest descent rule:
\begin{description}
\item[Beginning ]Start from any $v^0 \in V$ (For instance the one closest to the center in RNS)

\item[Body] For thecurrent $v^k \in V$, compute the directional derivative towards any adjacent vertex $v_j$:
\begin{itemize}
\item If $\frac{dg}{d\vec{\ell}} \geq 0$ towards any adjacent $v_j$ $\Longrightarrow$ $v^k$ is optimal $\Longrightarrow$ End
\item If $\frac{dg}{d\vec{\ell}} < 0$ towards some adjacent $v_j$  $\Longrightarrow$ Move towards the adjacent vertex $v^{k+1}$ with lowest derivative. $\Longrightarrow$ Repeat step with $v^{k+1}$.
\end{itemize}
\item[End] From the optimal vertex/face $\Longrightarrow$ find the optimal allocation.
\end{description}
\begin{theo}
The algorithm converges in a finite number of steps
\end{theo}
\begin{proof}
There is a finite number of vertices in the graph, and the algorithm cannot cycle.
\end{proof}


\begin{thebibliography}{99}


 
\bibitem{b00} Barbanel, Julius B. On the structure of Pareto optimal cake partitions. J. Math. Econom. 33 (2000), no. 4, 401--424.

\bibitem{b05} Barbanel, Julius B. The Geometry of Efficient Fair Division. Cambridge University Press (2005).

\bibitem{bz97} Barbanel, Julius B.; Zwicker, William S. Two applications of a theorem of Dvoretsky [Dvoretzky], Wald, and Wolfovitz to cake division. Theory and Decision 43 (1997), no. 2, 203--207.

 \bibitem{bt96} Brams, Steven J., Taylor, Alan D. Fair Division: From Cake-Cutting to Dispute Resolution. Cambridge University Press (1996).

 \bibitem{bt99} Brams, Steven J., Taylor, Alan D. The Win-Win Solution: Guaranteeing Fair Shares to Everybody. W.W.Norton (1999).

\bibitem{d01} Dall'Aglio, Marco. The Dubins–Spanier optimization problem in fair division theory. Journal of Computational and Applied Mathematics. 130 (2001), no.1, 17--40.

\bibitem{dd14} Dall'Aglio, Marco; Di Luca Camilla. Finding Maxmin Allocations in Cooperative and Competitive Fair Division. Annals of Operations Research. 230 (2014), no.1, 121--136.

\bibitem{dgk63} Donzer, L.; Gr\"unbaum, B.; Klee V. Kelly's Theorem and Its Relatives, Convexity Proc.\ Symp.\ Pure Math, American Mathematical Society (1963), 101--179 

\bibitem{k77} Kalai, E. Proportional Solutions to Bargaining Situations: Interpersonal Utility Comparisons. Econometrica. 45 (1977). no. 77. 1623--1630.

\bibitem{ks75} Kalai, E; Smorodinsky, M. Other Solutions to Nash's Bargaining Problem. Econometrica. 43 (1975). no. 3. 513--518.

\bibitem{w85} Weller; Dietrich. Fair Division of a Measurable Space. Journal of Mathematical Economics. 14 (1985). no.1. 5--17.






\end{thebibliography}
\end{document}